\documentclass{article}

\usepackage{arxiv}

\usepackage[utf8]{inputenc} 
\usepackage[T1]{fontenc}    
\usepackage{hyperref}       
\usepackage{url}            
\usepackage{array, booktabs}       
\usepackage{amsfonts}       
\usepackage{nicefrac}       
\usepackage{microtype}      

\usepackage{amsmath,amsthm,amsfonts,mathrsfs,anysize,fancyhdr,epsfig,mdframed, float,graphicx,dsfont,wrapfig}

\usepackage{chngcntr}
\usepackage{algorithm,algorithmic}
\usepackage{graphicx} \usepackage{grffile}
\usepackage[utf8]{inputenc} 
\usepackage[T1]{fontenc}    
\usepackage{hyperref}       
\usepackage{url}            
\usepackage{array, booktabs}       
\usepackage{amsfonts}       
\usepackage{nicefrac}       
\usepackage{amsmath,amsfonts,mathrsfs,epsfig,float,graphicx,dsfont,wrapfig}
\usepackage{subcaption,ragged2e}
\usepackage{autonum}

\usepackage{chngcntr}
\usepackage{algorithm,algorithmic}
\usepackage{amsopn}

\theoremstyle{plain}
\newtheorem{theorem}{Theorem}[section]

\newtheorem{lemma}[theorem]{Lemma}                              
\newtheorem{proposition}[theorem]{Proposition}

\usepackage{enumitem}
\newtheorem{definition}[theorem]{Definition}
\newtheorem{example}[theorem]{Example}
\newtheorem{remark}[theorem]{Remark}

\newtheorem{assumption}{Assumption}

\definecolor{pao_green}{rgb}{0.0, 0.5, 0.1}

\definecolor{brinkpink}{rgb}{0.98, 0.38, 0.5}

\newcommand{\revision}[1]{\textcolor{black}{#1}}

\newcommand{\bx}{\mathbf x}
\newcommand{\bz}{\mathbf z}
\newcommand{\bw}{\mathbf w}
\newcommand{\by}{\mathbf y}
\newcommand{\bd}{\mathbf d}
\newcommand{\bxi}{\boldsymbol{\xi} }

\newcommand{\cL}{\mathcal L}

\DeclareMathOperator{\R}{\mathbb{R}}
\DeclareMathOperator{\range}{range}

\newcommand{\argmin}{\arg\!\min}

\newcommand{\blambda}{\boldsymbol{\lambda}}
\newcommand{\bmu}{\boldsymbol{\mu}}

\newcommand{\inp}[2]{\ensuremath{\langle{#1},{#2}\rangle}}

\title{
Stochastic Mirror Descent for Convex Optimization with Consensus Constraints
}
\author{A. Borovykh, N. Kantas, P. Parpas,  G. A. Pavliotis}

\begin{document}

\maketitle 

\begin{abstract}
The mirror descent algorithm is known to be effective in situations where it is beneficial to adapt the mirror map to the underlying geometry of the optimization model.
However, the effect of mirror maps on the geometry of distributed optimization problems has not been previously addressed. 
\revision{In this paper we study an exact distributed mirror descent algorithm in continuous-time under additive noise.
We establish a linear convergence rate of the proposed dynamics for the setting of convex optimization.}
Our analysis draws motivation from the Augmented Lagrangian and its relation to gradient tracking. 
To further explore the benefits of mirror maps in a distributed setting we present a preconditioned variant of our algorithm with an additional mirror map over the Lagrangian dual variables. 
\revision{This allows our method to adapt to both the geometry of the primal variables, as well as to the geometry of the consensus constraint. 
We also propose a Gauss-Seidel type discretization scheme for the proposed method and establish its linear convergence rate.
For certain classes of problems we identify mirror maps that mitigate the effect of the graph's spectral properties on the convergence rate of the algorithm. 
Using numerical experiments we demonstrate the efficiency of the methodology on convex models, both with and without constraints. Our findings show that the proposed method outperforms other methods, especially in scenarios where the model's geometry is not captured by the standard Euclidean norm.}
\end{abstract}



\section{Introduction}
The choice of mirror map has a significant impact on both the theoretical and numerical performance of the Mirror Descent (MD) algorithm \cite{beck17book,bubeck2014convex,nemirovsky83}. 
With an appropriate choice of the mirror map, MD captures the geometry of the optimization model more faithfully than other first-order methods, and for certain classes of problems it is known to outperform other methods both in theory and in practice \cite{beck17book,bubeck2014convex}.  
There is an extensive literature on the mirror descent algorithm. However, the effect of the choice of the mirror map for \emph{distributed} optimization problems has received much less attention (see Section \ref{sec: previous work} for related work).
Distributed optimization problems, even when otherwise unconstrained, have to satisfy a consensus constraint. Existing algorithms do not capture the geometry of the consensus \revision{constraint}.
Motivated by the attractive theoretical and real-world performance of the mirror descent algorithm, in this paper we attempt to answer the following question: \emph{Does there exist a distributed variant of mirror descent that can accurately capture the geometry of distributed optimization models?} To answer this question, we propose a distributed mirror descent algorithm for the following optimization model, 
\begin{equation}\label{eq:distoptobj0}
f^\star=\min_{x^i\in\mathcal{X}} f(\bx)=\sum_{i=1}^N f_i(x^i), \ \ \ \text{s.t.} \ \cL\bx=0.
\end{equation}
\revision{Where $f_i$ denotes the objective function of node or particle $i$, $x^i\in\mathcal{X}\subset \R^d$, $i=1,\ldots,N$. 
When it is not necessary to distinguish between the particles we write $\bx^\top=[x^1 \ \ldots \ x^N]$ with $\bx\in \mathcal X^N$,
where $\mathcal X^N=\mathcal X\times\ldots\times\mathcal X$.}  
We assume that the particles communicate through a strongly connected, weighted, undirected graph $\mathcal{G}:=(V,E,A)$. We use $V$ to represent the nodes of the graph, $E$ its edges, $A$ is the weighted adjacency matrix and \revision{$\cL$ is the graph's Laplacian (see \ref{sec:preliminaries} for exact definitions).  Note that $\cL\bx=0\iff x_i=x_j, \forall (i,j)\in E$
}. We assume that each particle has access to its own objective function $f_i:\R^d\rightarrow \R$, and constraint set $\mathcal{X}\subset \R^d$. 
\begin{figure}[t]
\captionsetup[subfigure]{justification=Centering}
 \begin{subfigure}[t]{0.45\textwidth}
         \includegraphics[width=\textwidth]{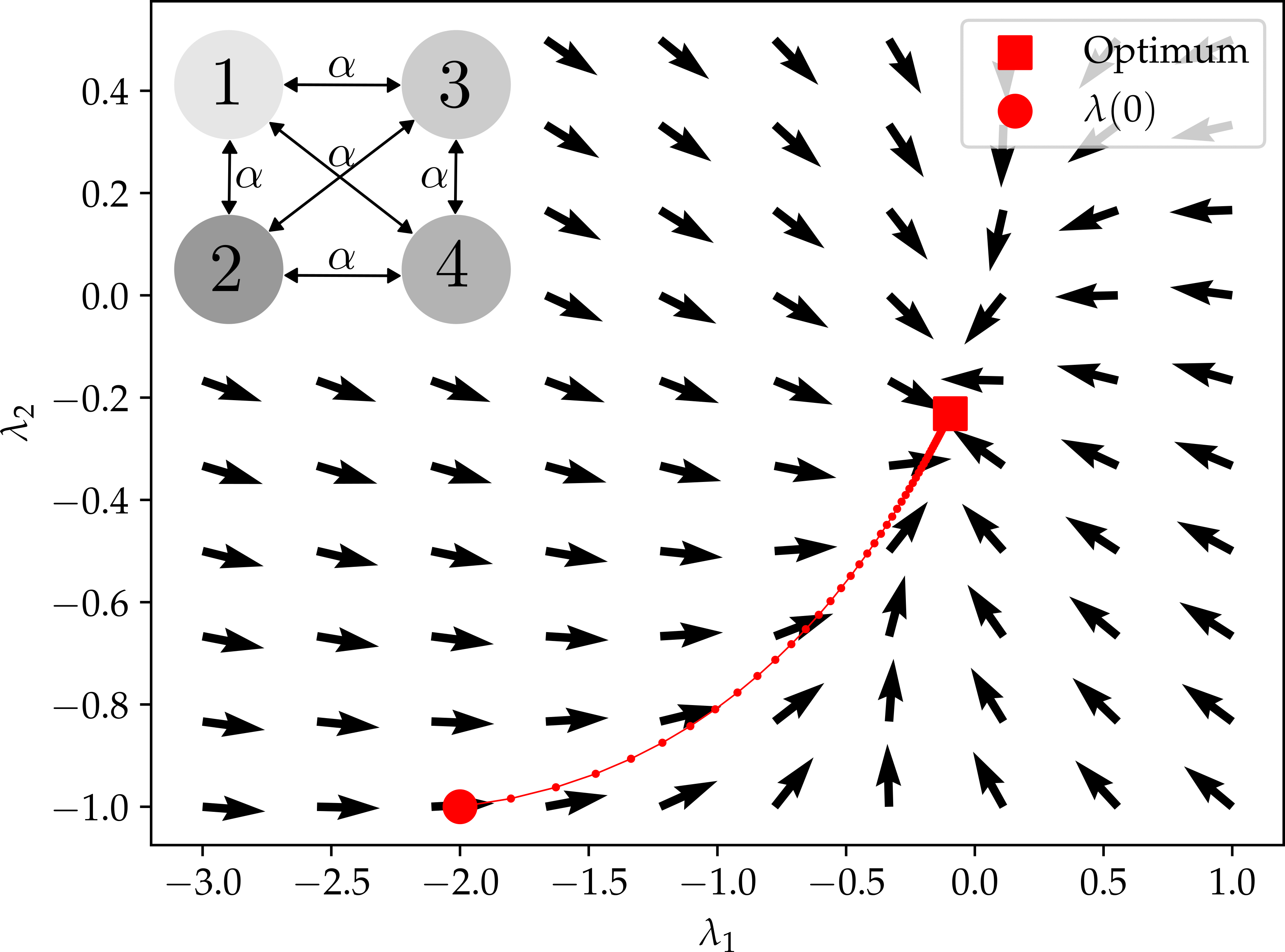}
         \caption{\revision{Dual ascent with a complete interaction graph $\alpha=1$.}}
         \label{fig:intro fully connected}
     \end{subfigure}\hspace{\fill} 
\begin{subfigure}[t]{0.45\textwidth}
         \centering
         \includegraphics[width=\textwidth]{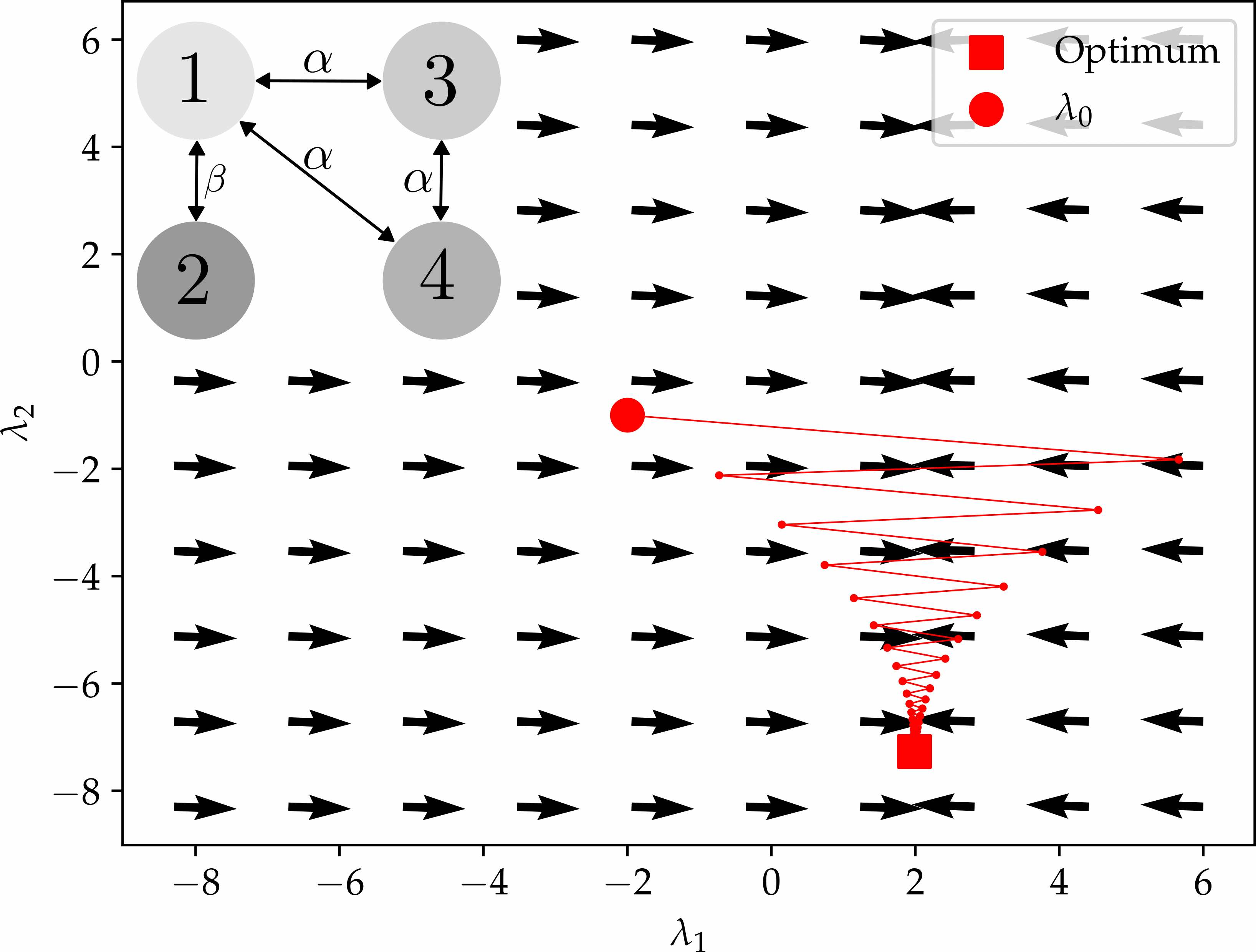}
         \caption{\revision{Dual ascent with an incomplete interaction graph $\alpha=1, \ \beta=0.1$.}}
         \label{fig:intro weakly connected}
     \end{subfigure}
\caption{
  \revision{
  Vector fields of the gradient of the dual function in \eqref{eq:intro dual function} for (a) a complete graph and (b) an incomplete graph. 
The two graphs are shown inset in the left-hand side corners of (a) and (b) respectively (for details see \ref{example intro}). In (a) the Laplacian of the graph is well-conditioned, in (b) the Laplacian is ill-conditioned. An example solution trajectory, obtained using an explicit discretization scheme, 
is shown in red for both cases. The explicit discretization scheme leads to oscillatory behavior and to eliminate it a small step-size needs to be used resulting in a slow convergence rate.}\label{fig:intro plots}}
\end{figure} 

\revision{
In mirror descent algorithms the mirror map is selected according to the geometry of \(\mathcal{X}\), and is independent of the consensus constraint. Unlike the $\mathcal X$ constraint, the consensus constraint couples all the particles together, and can therefore have a significant impact on the convergence properties of distributed optimization methods for \eqref{eq:distoptobj0}. The impact of the consensus constraint becomes apparent by examining the Lagrangian dual function of \eqref{eq:distoptobj0} given by,
\begin{equation}\label{eq:intro dual function}
q(\blambda)=\inf_{\bx\in\mathcal{X}^N} \left\{f(\bx)+\blambda^\top \cL\bx\right\}.
\end{equation}
Under the usual convexity assumptions strong duality holds i.e. $f^\star=\sup_{\blambda\in\mathbb{R}^{Nd}} \{q(\blambda)\}$ and $q(\blambda)$ is concave and differentiable (see e.g. \cite[Section 6.3]{bazaraa2013nonlinear}). Therefore, a natural starting point for a mirror descent algorithm is to generalize the dual-ascent method,
\begin{equation}\label{eq:dual ascent dynamics}
 \dot\blambda_t=\nabla_{\blambda} q(\blambda_t), 
\end{equation}
by introducing a mirror map for the Lagrange multipliers $\blambda$. The example below illustrates the challenges associated with a naive application of the dual ascent dynamics in \eqref{eq:dual ascent dynamics} to distributed optimization problems.
\begin{example}\label{example intro}
Consider a simple distributed optimization problem in the form of \eqref{eq:distoptobj0} where the objective function is given by the quadratic $\sum_{i=1}^4 (x^i-b^i)^2$, with $x^i\in\R$, and $b^i\in\R$ chosen randomly. To illustrate the importance of the spectral properties of $\cL$ we consider two different strongly connected graphs with $N=4$ nodes.
In \ref{fig:intro fully connected} we visualize a $2D$ projection of the vector field of $\nabla q(\lambda)$ for the fully connected graph shown inset. 
We obtained the $2D$ projection on $\lambda_1$ and $\lambda_2$ by fixing the other two Lagrange multipliers to their optimal value. 
The red line shows the path to a particular numerical solution of \eqref{eq:dual ascent dynamics} when the explicit Euler method is used to solve \eqref{eq:dual ascent dynamics}. 
Now consider the same problem but with a different strongly connected but incomplete graph. 
Such a graph is shown in the inset of \ref{fig:intro weakly connected}.
In this case the vector field of \eqref{eq:dual ascent dynamics} has completely different qualitative features. 
In \ref{fig:intro weakly connected} we observe that the optimal solution is along $\lambda_2=2$, and all paths reach this point after a \emph{rapid transient phase}. 
Again, the red line in \ref{fig:intro weakly connected} shows the trajectory of the explicit Euler discretization scheme applied to \eqref{eq:dual ascent dynamics}. 
The oscillations in the solution trajectory are a typical feature of stiff dynamical systems (see \cite[IV.1]{wanner1996solving}). 
It is well known that unless a small step-size is used it is common to observe oscillations in the numerical solution of stiff differential equations. 
Stiff dynamics present a significant challenge to the development of computationally efficient algorithms because optimization algorithms typically rely on explicit discretization schemes. The problem is even worse in the stochastic setting where constant step-sizes are used. We will revisit this example in \ref{sec:Discretization Analysis} and show that even when using an explicit discretization scheme the proposed algorithm converges in a single step independently of the graph's Laplacian.  
\end{example}
}

\revision{
Guided by the observations above, it is tempting to consider the following natural generalization of the Mirror Descent algorithm,
\begin{equation}\label{eq:intro dual md}
\partial_t\Psi(\blambda_t)=\nabla_{\blambda} q(\blambda_t).
\end{equation}
One could then select the mirror map for the dual variables, $\Psi$, in such a way as to better reflect the geometry of the interaction graph $\mathcal G$.
An additional advantage of such an approach is that we could apply existing results from the literature to establish both the convergence and convergence rate of \eqref{eq:intro dual md}. 
However, there are two significant challenges associated with the introduction of the mirror map $\Psi$ to the dual-ascent dynamics of \eqref{eq:dual ascent dynamics}. 
Firstly, computing the gradient of the dual requires the solution of the maximization problem in \eqref{eq:intro dual function}. 
Requiring the solution of an optimization problem in every iteration is only possible for problems with specific structure. 
Therefore using \eqref{eq:intro dual md} may be computationally infeasible for many problems. 
Secondly, even if the objective function is strongly convex the dual function in \eqref{eq:intro dual function} is concave and not strongly concave.
The dual function would be strongly concave if the Laplacian matrix had full rank. 
It is well known that the Laplacian matrix associated with a graph $\mathcal G$ does not have full rank (see \cite[Theorem 6.6]{bullo2019lectures}). 
Therefore applying existing convergence results to \eqref{eq:intro dual md} would only guarantee a sub-linear rate such as $O(1/T)$.
The objective of this paper is to address these challenges and to do so for the case where only noisy estimates of the gradients are available. 
Our results are based on a continuous-time analysis of distributed stochastic mirror descent dynamics from a dynamical systems perspective. Our contributions are summarized below. 
\begin{itemize}
    \item We propose a variant of the distributed mirror descent algorithm called Exact Preconditioned Interacting Stochastic Mirror Descent (EPISMD), that is able to converge exponentially fast to a neighborhood of the solution.   The neighborhood can be made arbitrary small in the case where there is no noise. 
    \item In order to avoid the computational challenges of computing the gradient of the dual we use an Augmented Lagrangian formulation. 
    For distributed gradient descent methods the links between the Augmented Lagrangian, gradient tracking and distributed optimization are well known.
    The novel aspect here is the analysis in the context of mirror descent methods where the link between the Augmented Lagrangian and the dynamics of mirror descent are less developed.
    \item We avoid the technical difficulty associated with the fact that the Laplacian matrix is positive semi-definite (and not positive definite) by showing that if the algorithm is initialized correctly then it will never visit a solution that is in the null space of the Laplacian matrix (see \ref{lemma: range space lambda}). 
    \item The analysis is performed under general assumptions regarding the choice of mirror maps. 
     We propose a Gauss-Seidel type discretization scheme and show that it too converges exponentially fast (Proposition \eqref{prop: discretization main}). We also identify suitable mirror maps that enable an improved convergence rate when the interaction graph is ill-conditioned (see \ref{sec:Discretization Analysis}).
    We note that the discretization analysis is based on the Lyapunov function identified from the continuous time analysis.
    \item We use the insights obtained from the properties of the dual function in order to propose a suitable mirror map for the Lagrange multipliers. See \ref{sec:choose precondtioner} for motivation, 
    \ref{sec:Discretization Analysis} for the analysis in discrete time and \ref{sec:num} for a numerical implementation. 
    \item In \ref{sec:num} we illustrate the performance of the proposed algorithm in constrained and unconstrained convex optimization problems. 
    The performance of the proposed method is in agreement with the theoretical results and also outperforms other state-of-the-art algorithms.
    When the graph is ill-conditioned then the proposed algorithm can be several orders of magnitude faster than other state-of-the-art methods.
\end{itemize}
}

\begin{table}[t]
    \centering
    \small
    \begin{tabular}{l|c|c|c|c}
         Reference& Mirror (Primal,Dual) & Linear Rate     &Noise & Constant step \\ \midrule\midrule
         Liang et al. \cite{liang2019exponential} & $(\times,\times)$ & $\times$ 
         &  $\times$ & $\times$  \\ \midrule      
         \cite{lin2016distributed, gharesifard2013distributed,liu2014continuous, zeng2016distributed} & $(\times,\times)$ & $\times$   & $\times$ & $\times$  \\ \midrule 
       
           \cite{di2016next, xu2015augmented} & $(\times,\times)$ & $\times$   & $\checkmark$ & $\checkmark$ \\ \midrule 
           Ram et. al.    \cite{ram2010distributed}
        & $(\times,\times)$ & $\times$   & $\checkmark$ & $\checkmark$ \\ \midrule 
         Shi et al. \cite{shi2015extra} & $(\times,\times)$ & $\checkmark$   & $\times$ & $\checkmark$ \\ \midrule 
        Qu \& Li \cite{qu2017harnessing} &$(\times,\times)$ & $\checkmark$   & $\times$ & $\checkmark$ \\ \midrule
        Pu \& Nedic \cite{pu2021distributed}, Sun et al.\cite{sun2019distributed} & $(\times,\times)$ & $\checkmark$   & $\checkmark$ & $\checkmark$ \\ \midrule 
         \midrule
        
        Duchi et al. \cite{duchi11}
        & $(\checkmark,\times)$ & $\times$ &   $\checkmark$ & $\checkmark$ \\ \midrule
        Nedic et al. \cite{nedic2015decentralized} & $(\checkmark,\times)$ &$\times$   & $\checkmark$ & $\times$ \\ \midrule
        Shahrampour et al. \cite{shahrampour2017distributed} & $(\checkmark,\times)$ & $\times$   & $\checkmark$ & $\times$ \\ \midrule
         Sun et. al. \cite{sun2020distributed,sun2021linear} & $(\checkmark,\times)$ & $\times$   & $\times$ & $\checkmark$ \\ \midrule 
        Sun et. al. \cite{sun2021centralized} & $(\checkmark,\times)$ & $\checkmark$   & $\times$ & $\checkmark$\\ \midrule \midrule
        \bf This work & $(\checkmark,\checkmark)$ & $\checkmark$  & $\checkmark$ & $\checkmark$ \\\midrule

 \end{tabular}
    \caption{\revision{Overview of related work. Note that linear convergence rate refers to a global linear rate.
}}
    \label{tab:convrates_2} \vspace{-1cm}
\end{table}        

\subsection{Previous work}\label{sec: previous work}
Distributed optimization has a variety of applications. 
A classic reference for distributed optimization is \cite{bertsekas2015parallel}, and more recent applications in statistical learning are described in \cite{boyd2011distributed}. The authors in \cite{bullo2019lectures} also describe several interesting applications. 
The literature on distributed optimization algorithms is vast. Since this paper focuses on exact distributed first-order algorithms for convex optimization models, we will focus on this class of algorithms. 
\revision{Note that in this context an exact algorithm is one that, in the absence of noise, converges to a solution of \eqref{eq:distoptobj0}}. 
Two algorithmic techniques can be used to develop exact distributed optimization algorithms. The first technique uses diminishing step-sizes, and the second one relies on gradient tracking. 
Gradient tracking is closely related to Augmented Lagrangian methods, however we remark that in the case of mirror maps the relationship between the two is less straightforward due to the mirrored variables. 
Algorithms with diminishing step-sizes tend to be very slow. 
So recent literature focuses on using constant step-sizes in combination with gradient tracking. 
\revision{ In Table \ref{tab:convrates_2}, we summarize selected related works that show how this paper fits within the existing literature. 
The top half of the table shows that there are many distributed algorithms that are based on the gradient descent method. 
Many of these works achieve a global linear convergence rate while also addressing the case of noisy gradients. 
As noted in the introduction there are many settings (e.g. optimization over the simplex) where mirror descent algorithms outperform projected gradient methods. 
For this reason there have been many attempts to generalize mirror descent to the distributed setting.
The bottom half of the table summarizes various recent works that study the mirror descent algorithm in the distributed setting. 
Among these only \cite{sun2021centralized} is an exact algorithm with linear convergence but does not deal with noisy gradient evaluations or mirroring of the Lagrangian dual variables. Moreover, the proof in \cite{sun2021centralized} is based on the solution of a Semi-Definite Programing (SDP) optimization model, and it is unclear if the proof technique can be extended to deal with noisy gradient evaluations.
Therefore compared to existing works, the algorithm we propose in this paper is developed in both continuous and discrete time, allows for additive Brownian noise, and achieves a global linear convergence rate, while being exact in the no-noise setting. 
Moreover, the mirror maps in the existing literature are used only to model the geometry of the separable constraints and/or the local objective functions and not the consensus constraint that is the main focus and distinctive feature of this paper. 
We achieve the latter by allowing for an additional mirror map for the Lagrange multipliers of the consensus constraint. For our analysis we use a Lyapunov function that suitably combines both the preconditioning of the primal and Lagrangian dual variables to attain non-asymptotic convergence results.
}

\subsection{Notation}
We use $\otimes$ to denote the Kronecker product, $I_d$ the $d$-dimensional identity matrix and $\mathbf 1_d$ denotes the $d$-dimensional vector of ones. $\textnormal{Diag}(a)$ with $a\in\mathbb{R}^d$ denotes a matrix with diagonal elements $[a_1,...,a_d]$. We use $A$ to denote the $N\times N$ weighted adjacency matrix associated with a graph $\mathcal G=(V,E)$. The graph Laplacian is given by $L:=\textnormal{Diag}(A\mathbf{1}_N)-A$ and we use the following notation $\mathcal{L}:=L\otimes I_d$ with $\mathcal{L}\in\R^{Nd\times Nd}$ to denote the vectorized version of the graph Laplacian. We use $\inp{x}{y}=x^\top y$ for the standard dot product, and $\inp{x}{y}_Q=\inp{x}{Qy}=x^\top Qy$ for the $Q$-inner product, for some positive definite matrix $Q$. \revision{We use $A\succeq B$ to denote that $A-B$ is positive semi-definite}. We assume that $\mathcal{X}\subseteq\mathbb{R}^d$ is a convex set.
We use $\mathcal D$ to denote an open set such that $\mathcal X\subset \text{cl} (\mathcal D)$. The set $\mathcal D$ will be used to denote the domain of the mirror maps of the mirror descent algorithm. We use $\mathcal X^*$ to denote the dual space of $\mathcal X$. The normal cone of $\mathcal X$ is defined as 
$N_{\mathcal X}(x)=\{z\in\mathcal X^*~|~ \inp{z}{y-x}\leq 0 ~\ \forall y\in\mathcal X\}$.

Given an arbitrary norm $||\cdot||$ on $\mathbb{R}^d$, we will define 
$B_{\|\cdot\|}:=\{v\in\mathbb{R}^d:\|v\|\leq 1\}$. The dual norm $\|\cdot\|_*$ is defined as $\|z\|_*:=\sup\{\inp{z}{v}:v\in B_{\|\cdot\|}\}$. 
If $A$ is a matrix then $\|A\|_2$ denotes its spectral norm and 
we assume that the dual norm is compatible with the spectral norm, i.e. $\|Az\|_*\leq \|A\|_2\; \|z\|_*$. We will make use of the following generalized Cauchy inequality,
\revision{
\begin{align}
|\inp{v}{w}| \leq \|v\|_*\|w\| \ \ \ \forall w\in \mathcal X, \ v\in\mathcal X^*.
\end{align}
Since $0\leq (\|v\|_*-\|w\|)^2=\|v\|^2_*+\|w\|^2-2\|v\|_*\|w\|$, we also have,
\begin{equation}\label{eq: bound inp}
|\inp{v}{w}|\leq \frac12 \|v\|_*^2+\frac12 \|w\|^2.
\end{equation}
}

A function $g$ is said to be $L$-Lipschitz continuous with respect to a norm 
$\|\cdot\|$ if $\|g({x})-g({y})\|\leq L\|{x}-{y}\|, \ \forall x,y\in\mathcal X$.  The Bregman divergence associated with a convex, differentiable function $g~:\:\mathcal{X}\rightarrow\mathbb{R}$ is defined as follows,
\begin{align}
    D_g({x},{y})=g({x})-g({y})-
    \inp{\nabla g({y})}{{x}-{y}}.
\end{align}
If the second-order derivative of $g$ exists it furthermore holds,
\begin{align}
    &\nabla_{x} D_g({x},{y}) = \nabla g({x})-\nabla g ({y}), \;\; \nabla_{y} D_g({x},{y}) = \nabla^2 g({y})({y}-{x}).\label{eq:breg1}
\end{align}
\revision{The aggregate cost function will be written as $f(\bx)=\sum_{i=1}^N f_i(x^i),$ where $\mathbf{x}=[{x^1},\ldots,{x^N}]^T$ denotes the stacked vector of particles and each $x^i\in\mathcal{X}$. 
For a function $f:\mathcal X^N\mapsto \R$, unless specified otherwise gradient vectors $\nabla f$ are taken with respect to the joint particle vector $\mathbf{\bx}$ following the convention that $\nabla f(\bx)^\top=[\nabla f_1(x_1) \ \nabla f_2(x_2) \ \ldots \ \nabla f_N(x_N)]$  and for Hessian matrices we write $\nabla^2 f(\bx)=\text{diag}(\nabla^2f_1(x_1),\nabla^2f_2(x_2), \ \ldots , \nabla^2 f_N(x_N))$. We will use $(X^\star,\Lambda^\star)$ to denote the set of primal-dual variables that satisfy the first order optimality conditions for \eqref{eq:distoptobj0}.}

\begin{remark}\label{remark:mixed norms}
The space of the Lagrange multipliers for the consensus constraint, $\blambda\in\Lambda\subset \R^{Nd}$, will play an important role in the definition of the algorithms below. 
We note that the norm associated with $\blambda\in\Lambda$ will not necessarily be the same as the one used for the primal variables $\bx\in\mathcal{X}^N$. 
We will however use the same notation: $\|\cdot\|$, and its dual $\|\cdot\|_*$ for both spaces, and it will be clear from context which norm is being used. For $\bw=[\bx^\top, \ \blambda^\top]^\top$ we will use the following mixed norm convention $\|\bw\|=\|\bx\|+\|\blambda\|$, with the understanding that the two norms could be different. For example, the norm in $\mathcal X^N$ could be the $\ell_1$, and in $\Lambda$ the $Q$-norm (for some positive definite matrix $Q$) so that,
$\|\bw\|=\|\bx\|_1+\|\blambda\|_Q.$
\end{remark}

\color{black}
\section{Distributed Stochastic Mirror Descent: Exact and Preconditioned Dynamics}\label{sec: algorithms}
In this section we introduce the proposed distributed MD algorithm.
We adopt a dynamical systems point of view for our analysis. 
The discretization of the proposed method is discussed in \ref{sec:Discretization Analysis} and we report on numerical experiments in \ref{sec:num}.
For an introduction to the original mirror descent algorithm we refer the interested reader to~\cite[Ch. 9]{beck17book}. 
We also remark that Langevin dynamics with mirror maps have been used in the sampling literature \cite{zhang2020wasserstein, chewi2020exponential, li2022mirror} but the the primary emphasis in this paper is on optimization.

The most widely studied algorithm for distributed Mirror Descent is the Interacting Stochastic Mirror Descent (ISMD) algorithm (also known as distributed dual averaging). In the continuous time setting the dynamics of ISMD are as follows,
\begin{equation}\label{eq:cmdpart}
dz_t^i = -\eta\nabla f_i(x_t^i) dt + \epsilon\sum_{j=1}^N A_{ij}(z_t^j-z_t^i) dt + \sigma dB_t^i,\quad x_t^i=\nabla\Phi^*(z_t^i),
\end{equation}
for particles $i=1,...,N$, and where $B_t^i$ are independent Brownian motions. The matrix $A=\{A_{ij}\}_{i,j=1}^N$ is an $N\times N$ doubly-stochastic matrix representing the interaction weights and $\eta,\epsilon$ are tuning constants representing the learning rate and interaction strength, respectively. 
For simplicity in most of the subsequent analysis we set $\eta=\epsilon=1$, but it is straightforward to extend the results to arbitrary values of $\eta$ and $\epsilon$.
In the context of modern large scale applications, we note that understanding convergence under the presence of noise is often motivated from computational considerations such as when sub-sampling the gradient of $f$ or sub-sampling the interaction graph. 

Using the graph Laplacian, we can rewrite the evolution in vector form as,
\begin{align}\label{eq:cmdpart_vec}
d\mathbf{z}_t =\left( - \nabla f(\mathbf{x}_t)-\mathcal{L}\mathbf{z}_t\right)dt + \sigma d\mathbf{B}_t, \;\; \mathbf{x}_t = \nabla \Phi^*(\mathbf{z}_t),
\end{align}
where $\mathbf{B}_t:=((B_t^1)^T,...,(B_t^N)^T)^T$. If $\Phi=I$ then the algorithm reduces to the standard Distributed Gradient Descent (DGD) in continuous time. This algorithm was proposed in \cite{raginsky12} but its convergence was only established for strongly convex quadratic functions. 
The case where all the functions are strongly convex and identical (i.e. $f_1=\ldots=f_N$), but not necessarily quadratic, was analyzed in \cite{borovykh20b}.
The discrete time version of the algorithm for the general convex case was discussed in \cite{duchi11}, but exact convergence was only established under a diminishing step-size strategy. Since we could not find a convergence analysis of \eqref{eq:cmdpart_vec} we provide a convergence proof in \ref{app:auxiliaryresults} that shows that the dynamics of \eqref{eq:cmdpart_vec} will converge to an approximate solution of \eqref{eq:distoptobj0}. 

It is known that DGD will not converge to the exact solution of \eqref{eq:distoptobj0}. 
When $\sigma=0$ and $\Phi=I$ then this observation follows from the fact that the fixed points of \eqref{eq:cmdpart_vec} do not coincide with the first-order necessary and sufficient conditions for optimality for the optimization problem in \eqref{eq:distoptobj0}; see 
\cite{shi2015extra,shi2015network} for more details in this case. 
In \ref{sec:first order fails} we adapt arguments of \cite{shi2015extra,shi2015network} to show that the dynamics in \eqref{eq:cmdpart_vec} also fail to converge to the exact solution of \eqref{eq:distoptobj0} and that the convergence of \eqref{eq:cmdpart_vec} to the exact solution of \eqref{eq:distoptobj0} can only occur under additional and restrictive assumptions on $f$. 
In simple terms, our result indicates that one cannot make the dynamics of DGD exact just by an appropriate choice of mirror map.

\subsection{Exact Preconditioned Interacting Stochastic Mirror Descent (EPISMD)}
The ISMD algorithm described above has two limitations: (1) even if $\sigma=0$ the stable points of \eqref{eq:cmdpart_vec} do not coincide with the ones of \eqref{eq:distoptobj0} and (2) the mirror map $\Phi$ is only chosen to reflect the geometry of $\mathcal X$ and not the consensus constraint.
To address the first limitation we propose to exploit the links between the constrained optimization problem in \eqref{eq:distoptobj0} and its Augmented Lagrangian,
\begin{equation}\label{eq:al difinition}
L(\bx,\blambda)=f(\bx)+\inp{\blambda}{\cL\bx}+\frac12\|\cL^\frac12\bx\|_2^2.
\end{equation}
Since we are in a convex setting we have by strong duality that $f^\star=\max_{\blambda}\min_{\bx} L(\bx,\blambda)$. 
To address the first shortcoming it is entirely natural to perform a descent step on \eqref{eq:al difinition} w.r.t $\bx$ and an ascent step w.r.t the Lagrange multipliers
$\blambda$. 
Note that the Augmented Lagrangian in the form of \eqref{eq:al difinition} is preferred to the classical Lagrangian because of the regularizing effect of the quadratic term during the ascent/descent steps (see \cite[Chapter 9]{bazaraa2013nonlinear}). To address the second shortcoming we use two mirror maps. The first mirror map $\Phi$, is used to reflect the geometry of the primal space $\mathcal X$. Then, a second mirror map, $\Psi$, is used to adapt to the geometry associated with the consensus constraint in \eqref{eq:distoptobj0}.
We use $\bmu$ to denote the dual of the Lagrange multiplier, i.e. $\bmu=\nabla \Psi(\blambda)$. 
Figure \ref{fig:dmd} explains the main steps in mirror descent with the two maps. At time-step $t$ the primal-dual pair $(\bx_t,\blambda_t)$ is mapped to 
$(\bz_t,\bmu_t)=(\nabla \Phi(\bx_t),\nabla \Psi(\blambda_t))$. 
The inverse $(\nabla \Phi^{-1},\nabla \Psi^{-1})$ maps the duals back to the primal space $\mathcal X\times\Lambda$.
The proposed algorithm, is given below,
\begin{equation}\label{eq:alm continious sde prec}
\begin{split}
 &   d\mathbf{z}_t=-\nabla f(\mathbf{x}_t)dt-\mathcal L\mathbf{x}_tdt - \mathcal L\blambda_tdt +\sigma d\mathbf{B}_t,  \quad \bx_t=\nabla\Phi^*(\bz_t)\\
 &   d\bmu_t=\mathcal L\mathbf{x}_t dt, \quad \blambda_t=\nabla\Psi^*(\bmu_t).
\end{split}
\end{equation}
The motivation for the dynamics above stems from the Bregman Augmented Lagrangian. 
One can view the dynamics in \eqref{eq:alm continious sde prec} as performing a descent step with respect to primal variables $\bx$ and an ascent step with respect to the Lagrangian dual variables.
To see the idea behind this intuition, we note that in the noise-free case ($\sigma=0$) we can characterize \eqref{eq:alm continious sde prec} as a limiting case of the classical proximal algorithm of \cite{censor1992proximal}
\[
\begin{split}
 &   {\bz}_{t+\delta} = \inf_{\by}\left\{L(\by,\lambda_t)+\frac{1}{\delta}D_\Phi(\by,\bx_t)\right\},\\
 &   \bmu_{t+\delta} = \sup_{\mathbf{\nu}}\left\{ L(\bx_t,\mathbf{\nu})-\frac{1}{\delta}D_\Psi(\mathbf{\nu},\blambda_t)\right\}.
\end{split}
\]
Applying the first order optimality conditions to the two optimization problems above and taking formally the limit $\delta\downarrow0$ we can derive \eqref{eq:alm continious sde prec} with $\sigma=0$; see \cite{alvarez2004hessian} for a more rigorous exposition. 

\begin{figure}[t]
     \centering
     \includegraphics[width=0.8\textwidth]{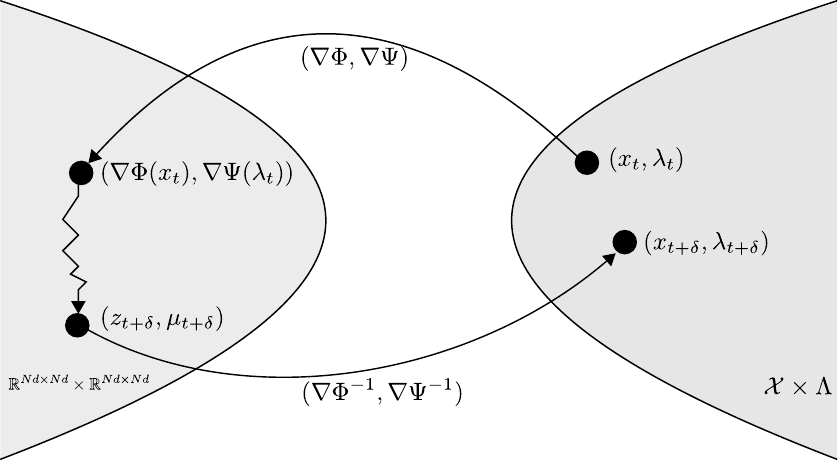}
      \caption{Stochastic Mirror Descent with two mirror maps. $\Phi$ maps the primal variables to the dual space, and $\Psi$ maps the Lagrangian dual variables associated with the consensus constraint to the dual of the Lagrange multipliers.}
        \label{fig:dmd}
\end{figure}

\color{black}
\section{Convergence Analysis of Distributed Stochastic Mirror Descent: Exact and Preconditioned Dynamics}
\revision{
In this section we show that the dynamics of EPISMD in \eqref{eq:alm continious sde prec} converge exponentially fast to an area of the optimum.
This area depends on $\sigma$. If $\sigma=0$ the algorithm converges to the exact solution of \eqref{eq:distoptobj0}.
We start this section by collecting all our assumptions and some preliminary results that will be useful later on.
}

\subsection{Preliminaries and Assumptions}\label{sec:preliminaries}
\revision{
We split our assumptions into three broad categories. 
We start in \ref{sec:optimilaity conditions} with our assumptions regarding the optimization model.
In \ref{sec: network assumptions} we state our assumptions regarding the network, and finally in \ref{sec:mirror maps} we discuss our assumptions for the mirror maps.
These assumptions will hold throughout the paper. We make some additional assumptions in \ref{sec:Discretization Analysis} when we introduce our discretization scheme.
We also state some technical lemmas that will be useful for the analysis of the method and whose proofs are presented in \ref{app:auxiliaryresults}.
}
\subsubsection{Optimality Conditions and Model Assumptions}\label{sec:optimilaity conditions} 
The consensus constraint in \eqref{eq:distoptobj0} is satisfied if and only if $\cL\bx=0$, where $\cL$ denotes the vectorized graph Laplacian. Therefore the optimality conditions for \eqref{eq:distoptobj0} are as follows,
\revision{
\begin{align}
-\nabla f(\bx^\star)-\cL\blambda^\star&\in N_{\mathcal X^N}(\bx^\star),\\
\cL\bx^\star& =0.
\end{align}
}
We make the following assumptions regarding the objective function and solution set.
\begin{assumption}\label{ass:f}  Each $f_i$ in \eqref{eq:distoptobj0} is proper, convex and twice differentiable. The elements of $X^\star$ are in the interior of $\mathcal X$.
\end{assumption}
If the solution of \eqref{eq:distoptobj0} is in the interior of $\mathcal {X}$, and if $f$ is convex, then we must have that $N_{\mathcal X}(\bx^\star)=\{\mathbf 0\}$ for any $\bx^\star\in X^\star$. Because the focus of this paper is on the effect of the consensus constraint and its impact on the dynamics of the algorithm, we will assume that the optimal solution of \eqref{eq:distoptobj0} is in the interior of $\mathcal X$. 
For certain applications, especially in machine learning, the assumption that the solution lies in the interior of the feasible set holds (e.g. \cite{mertikopoulos18,sra2012optimization}). 
The extension to the general case requires modifications to our convergence analysis similar to \cite{mertikopoulos18} and will be an interesting extension to our work. 
The derivatives of $f$ are required so that we can apply It\^o's formula. 
\begin{definition}
\revision{
We say that $f:\mathcal X^N\rightarrow \R$ is $\mu$-strongly convex w.r.t. some norm $||\cdot||$ provided that $ \|\nabla f(\mathbf{x})-\nabla f(\mathbf{y})\|_* \geq \mu \|\mathbf{x}-\mathbf{y}\|$. Similarly, a function $f$ is $L$-smooth w.r.t. some norm $\|\cdot\|$ when
$\|\nabla f(\mathbf{x})-\nabla f(\mathbf{y})\|_*\leq L\|\mathbf{x}-\mathbf{y}\|$.
}
\end{definition}
Some of our results will use the notion of relative strong convexity and smoothness.
We refer the reader to \cite{lu2018relatively}. Below we present some definitions and properties that will be useful later on.
\begin{definition}[Relative strong convexity]
A function $g:\mathcal X^N\rightarrow \R$ is $\mu$-strongly convex with respect to some convex function $h$ if for any $\bx,\by\in\mathcal X^N$ the following holds,
$$g(\mathbf{x})\geq g(\mathbf{y})+\nabla g(\mathbf{y})^T(\mathbf{x}-\mathbf{y})+\mu D_h(\mathbf{x},\mathbf{y}).$$
Or equivalently, 
$
\inp{\mathbf{x}-\mathbf{y}}{\nabla g(\mathbf{x})-\nabla g(\mathbf{y})}\geq 
\mu \inp{\mathbf{x}-\mathbf{y}}{\nabla h(\mathbf{x})-\nabla h(\mathbf{y})}.
$
\end{definition}
\begin{definition}[Relative smoothness]
A function $g:\mathcal X^N\rightarrow \R$ is $\alpha$-smooth with respect to some function $h$ if for any $\bx,\by\in\mathcal X^N$ the following holds,
$$g(\mathbf{x})\leq g(\mathbf{y})+\nabla g(\mathbf{y})^T(\mathbf{x}-\mathbf{y})+\alpha D_h(\mathbf{x},\mathbf{y}).$$
Or equivalently, 
$\inp{\mathbf{x}-\mathbf{y}}{\nabla g(\mathbf{x})-\nabla g(\mathbf{y})}\leq \alpha \inp{\mathbf{x}-\mathbf{y}}{\nabla h(\mathbf{x})-\nabla h(\mathbf{y})}$.
\end{definition}
If we assume that $g$ is $\mu$-strongly convex and $\alpha$-smooth with respect to $h$ it holds,
\begin{align}\label{eq:smoothnessandmore}
    \mu D_h(\mathbf{x},\mathbf{y})\leq D_g(\mathbf{x},\mathbf{y})\leq 
    \alpha D_h(\mathbf{x},\mathbf{y}).
\end{align}
We adopt the following definition for the convex conjugate of a relatively strong convex function.
\begin{definition}[Convex conjugate]
Let $g:\mathcal X^N\rightarrow \R$ be a $\mu$-strongly convex function with respect to some $h$. Then
$g^*(\mathbf{z}) := \max_{\mathbf{x}\in\mathcal{X}^N} \{ \inp{\mathbf{z}^T}{\mathbf{x}}-g(\mathbf{x})\}$
is its Legendre-Fenchel convex conjugate.
When $g$ is differentiable, we also have $\nabla g^*(\mathbf{z}):=\arg\max_{\mathbf{x}\in\mathcal{X}^N}
\{  \inp{\mathbf{z}^T}{\mathbf{x}}-g(\mathbf{x})\}$, and $\nabla g \circ\nabla g^*(\mathbf{z})=\mathbf{z}$. 
\end{definition}

\subsubsection{Network Assumptions}\label{sec: network assumptions}
We first state our assumptions on the network topology. 
\begin{assumption}\label{ass:graph}
The graph $\mathcal{G}$ is connected, undirected and the adjacency matrix $A$ is doubly stochastic.
\end{assumption} 
These assumptions imply that the graph Laplacian $\mathcal{L}$ is a real symmetric matrix with nonnegative eigenvalues. We will denote the pseudo-inverse of $\mathcal{L}$ by $\mathcal{L}^+$ such that, 
\begin{equation}
\mathcal{L}\mathcal{L}^+\mathcal{L}=\mathcal{L}.  \label{eq:pseudoinverse}
\end{equation}
We will use the following definition of the $\beta$-regularized Laplacian \cite{bullo2019lectures},
\begin{equation}\label{eq:reglapl}
    \mathcal{L}_\beta = \mathcal{L}+\frac{\beta}{N}\mathbf{1}_N\mathbf{1}_N^\top\otimes I_d.
\end{equation}
In the remainder we will assume $\beta>0$ as this results in $\mathcal{L}_\beta$ being positive definite. We define the  Rayleigh quotient associated with $\mathcal{L}_\beta$ as follows,
\begin{equation}\label{eq: Rayleigh quotient Reg Laplacian}
\kappa_{\beta,N} = \max_{\mathbf d_x\in\mathbb{R}^{Nd}\setminus\{0\}}
  \ \frac{\|\mathcal{L_\beta}^{\frac{1}{2}}\mathbf d_x\|_2^2}{\|\mathbf d_x\|_2^2}.
\end{equation}
It holds that \cite[p.116]{bullo2019lectures},
\begin{equation}
\label{eq:reglaplacianpseudo}
    \mathcal{L}_\beta^{-1} = \mathcal{L}^++\frac{1}{\beta N}\mathbf{1}_N\mathbf{1}_N^\top\otimes I_d \succeq \mathcal{L}^+,
\end{equation}
where the latter inequality follows from the fact that $\mathbf{1}_d\mathbf{1}_d^\top\otimes I_N$ is positive semidefinite. 
In \ref{sec: convergence} we will use the following technical Lemma. 
\begin{lemma}\label{lemma: bound laplacian}
 Let Assumption \ref{ass:graph} hold and suppose that $\kappa_{\beta,N}$ is as defined in 
 \eqref{eq: Rayleigh quotient Reg Laplacian} then,
 \begin{equation}
 \inp{\bx}{\mathcal L\bx}\geq 
 \frac{1}{\kappa_{\beta,N}}
 \|\mathcal L \bx\|_2^2.
 \end{equation}
\end{lemma}

\subsubsection{Mirror Maps}\label{sec:mirror maps}
In this section we state our assumption for the two mirror maps $\Phi$ and $\Psi$. 
These assumptions are fairly standard in the literature of the mirror descent algorithm, see for example \cite[Chapter 9]{beck17book}. 
We also state some identities and relationships that will be useful later on.

\begin{assumption}[Mirror map]\label{ass:mirror1}
$\Phi:\mathcal D\rightarrow \mathbb R$ is proper, twice differentiable, $\mu_\Phi$-strongly convex and $L_\Phi$-smooth w.r.t. some norm $||\cdot||$. The same holds for $\Psi:\Lambda\rightarrow\mathbb R$ with constants $\mu_\Psi$ and $L_\Psi$. respectively.
We furthermore make the additional assumption 
$\nabla\Phi^*(\mathbb{R}^d)=\mathcal{X}$,  $\nabla\Psi^*(\mathbb{R}^d)=\Lambda$, $\Phi^*$ is $L_{\Phi^*}$-smooth and assume uniform boundedness of the Laplacians of $\Phi^*,\Psi^*$ such that $||\Delta\Phi^*||_\infty,||\Delta\Psi^*||_\infty<\infty$.
\end{assumption}
The assumption that $\nabla\Phi^*$ maps directly to $\mathcal{X}$ (and similarly for $\Psi$) avoids the need for projections. Extending our results without this assumption is possible by following a route similar to \cite{mertikopoulos18}. A common mirror map that satisfies our assumptions for the primal variables is the negative entropy $\Phi(x)=\sum_{i=1}^d x_i\ln(x_i)$.
For the Lagrange multipliers a $Q$-norm may be more appropriate i.e. $\Psi(\lambda)=\frac12\lambda^\top Q\lambda$.
The assumptions that  $||\Delta\Phi^*||_\infty$, and $||\Delta\Psi^*||_\infty$ are bounded from above are needed in order to apply It\^o's lemma to the Lyapunov function we will analyze in \ref{sec: convergence}.

A useful property of the Bregman divergence induced by mirror maps that satisfy our assumptions is the following,
\begin{equation}
D_{\Phi^*}(z,z')=D_\Phi(x',x),\label{eq:breg7}
\end{equation}
where $z=\nabla\Phi(x)$ and $z'=\nabla\Phi(x')$.
For $x,y,z\in\mathbb{R}^d$ we have the triangle property for Bregman divergences (see \cite[Lemma 9.11]{beck17book})
\begin{align}\label{eq:triangle_bregman}
    \langle x-y, \nabla\Phi(z)-\nabla\Phi(y)\rangle = D_\Phi(x,y)+D_\Phi(y,z)-D_\Phi(x,z).
\end{align}
We also make use of the following property,
\begin{equation}\label{eq:symmetry_Bregman}
D_f(\bx,\bx')\leq \alpha(\Phi) D_\Phi(\bx',\bx),
\end{equation}
where $\alpha(\Phi)= \frac{L_fL_\Phi}{\mu_\Phi}$. This property follows from the relative smoothness assumption combined with the strong convexity and Lipschitz assumption on $\Phi$,
\begin{align}
D_f(\bx,\bx')\leq L_f D_\Phi(\bx,\bx')\leq\frac{ L_fL_\Phi}{2}\|\bx'-\bx\|^2
\leq \frac{ 2 L_fL_\Phi}{2\mu_\Phi}D_\Phi(\bx',\bx),
\end{align}
where we used the notation $\Phi(\bx)=\sum_{i=1}^N\Phi(x^i)$.
We will use the following Rayleigh quotient,
\begin{equation}\label{eq: Rayleigh quotient l2}
\kappa_N = \max_{\mathbf d_x,\mathbf d_\lambda\in\mathbb{R}^{Nd}\setminus\{0\}} \ 
\max\left\{\frac{\|\cL^\frac12 \mathbf d_x\|^2_2}{\|\mathbf d_x\|^2},\frac{\|\cL^\frac12 \mathbf d_\lambda\|^2_2}{\|\mathbf d_\lambda\|^2}\right\}.
\end{equation}
Note that the norms for $\mathbf d_x$ and $\mathbf d_\lambda$ in the definition above may be different (see Remark \ref{remark:mixed norms}).
%
We will also make use of following generalized Rayleigh quotient,
\begin{equation}\label{eq: General Rayleigh quotient}
\kappa_g = \inf_{\bx\in\mathcal{X}^N,\mathbf d_x,\mathbf d_\lambda\in\mathbb{R}^{Nd}\setminus\{0\}} 
\ \frac{\|A(\bx)[\mathbf d_x ^T \   \mathbf d_\lambda^T]^\top\|^2
_{\nabla^2 \Phi^*(\bz)}}{\|\mathbf d_x\|^2+\|\mathbf d_\lambda\|^2},
\end{equation}
where $A(\bx)=[\nabla^2f(\bx)+\mathcal L , \mathcal L]\in\mathbb R^{Nd\times 2Nd}$ and the notation $[X,Y]$ is used to denote a concatenation of matrices.
%
\begin{lemma}\label{lemma: postive k_g}
Suppose Assumptions \ref{ass:f} to \ref{ass:mirror1} hold and that $f$ is relatively strongly convex with respect to $\Phi$, then $\kappa_g$ defined in \eqref{eq: General Rayleigh quotient} is positive. 
\end{lemma} 
This result is not obvious since $A(\bx)$ is not a square matrix, and the norm used in the definition of \eqref{eq: General Rayleigh quotient} is not standard; a short proof is provided in \ref{app:auxiliaryresults}.

\revision{
Lastly, we will need the following relationship between the gradient of the Augmented Lagrangian and the two mirror maps.
}
\begin{lemma}\label{lemma:bound on lagrangian and mirror map v1}
Suppose Assumptions \ref{ass:f} to \ref{ass:mirror1} hold. Then for an arbitrary optimal primal dual pair $(\bx^\star,\blambda^*)$ we have
\begin{align}
\begin{split}
\|\nabla f(\bx)+\cL\blambda+\cL \bx\|_{\nabla\Phi^*(\bz)} 
&\geq \frac{2\kappa_g}{\hat\mu}
\left(\sum_{i=1}^N D_\Phi(x^\star,x^i)+D_\Psi(\lambda^\star,\lambda^i)
\right)\\
&= \frac{2\kappa_g}{\hat\mu}\left(\sum_{i=1}^N D_{\Phi^*}(z^i,z^\star)
+D_{\Psi^*}(\mu^i,\mu^\star)\right),
\end{split}
\end{align}
where $\hat\mu=\min\{\mu_\Phi,\mu_\Psi\}$
\end{lemma} 
\revision{The lemma above suggests that making progress towards reducing the norm of the gradient of the Augmented Lagrangian also means that we will be getting closer to the optimal solution. This is not surprising and it is a generalization of a similar inequality for the gradient descent method, see e.g. \cite[Chapter 2]{nesterov2018lectures}}

\subsection{Convergence Analysis}\label{sec: convergence}
Ler $(\bx^\star,\blambda^\star)=(\mathbf 1_N\otimes x^\star,\mathbf 1_N\otimes\lambda^\star)$ be a point that satisfies the first order optimality conditions for \eqref{eq:distoptobj0}. Note that when $\sigma=0$, $(\bx^\star,\blambda^\star)$ is also a fixed  point of \eqref{eq:alm continious sde prec}. In the case that $f$ is only convex and thus multiple minimizers might exist, then we define the optimal primal dual pair, $(x^\star,\lambda^\star)$ as follows,
\begin{equation}\label{eq:define optimal pair}
(x^\star,\lambda^\star)=\argmin_{x,\lambda\in(X^\star,\Lambda^\star)}\left\{
D_\Phi(x,x_0)+D_\Psi(\lambda,\lambda_0)\right\},
\end{equation}
where $x_0,\lambda_0$ are the initial point of the algorithm. 

The connection of the dynamics of \eqref{eq:alm continious sde prec} and the Augmented Lagrangian is key to the convergence analysis below. The analysis of the algorithm in \eqref{eq:alm continious sde prec} is based on the following Lyapunov function,
\begin{equation}\label{eq:laypunov exact primal v2}
V(\bx,\blambda)=c(V_1(\bx)+V_2(\blambda))+V_3(\bx,\blambda),
\end{equation}
where,
\begin{align}\label{eq:lyapunov v1_3}
&V_1(\bx)=\sum_{i=1}^N D_\Phi(x^\star,x^i),\;\; V_2(\blambda)=\sum_{i=1}^N D_\Psi(\lambda^\star,\lambda^i),\\
&V_3(\bx,\blambda)=D_f(\bx,\bx^\star)+\inp{\bx-\bx^\star}{\cL(\blambda-\blambda^\star)}+
\frac12\|\cL^{\frac12}\bx\|^2_2,
\end{align}
and $c>0$ will be specified below.
\revision{
The role of $V_1$ and $V_2$ are clear, they measure the distance between the point $(\bx,\blambda)$ and the optimal Karush–Kuhn–Tucker (KKT) primal-dual pair $(\bx^\star,\blambda^\star)$.
The role of $V_3$ becomes more obvious once re-arranged as follows,
\[
V_3(\bx,\blambda)=L(\bx,\blambda)-L(\bx^\star,\blambda^\star).
\] 
Therefore $V_3$ is also a measure of optimality since $V_3(\bx^\star,\blambda^\star)=0$. Unfortunately, neither $V_1+V_2$ or $V_3$ on their own are Lyapunov functions for the stochastic dynamics in \eqref{eq:alm continious sde prec}. 
Instead the proof below balances the two terms using the constant $c$. Note that given our assumptions on $f$, $\Phi$ and $\Psi$ it follows that $V$ is a proper function (has compact sub-level sets).} 

First we establish useful upper and lower bounds on \eqref{eq:laypunov exact primal v2}.
\revision{
Note that the Lemma below is enough to establish the convergence of the algorithm but not the actual rate that $V$ dissipates. 
This is one of the reasons that the Lyapunov function in \eqref{eq:laypunov exact primal v2} is different than other works (e.g. \cite{sun2020distributed}) that establish the convergence of distributed mirror descent but not the convergence rate for $V$.
}

\begin{lemma}\label{lemma: bound lyapunov 2}
Let Assumptions \ref{ass:f}-\ref{ass:mirror1} hold. 
Then $V_t$ in \eqref{eq:laypunov exact primal v2} satisfies the following, 
\begin{description}
    \item[(I)] $V(\bx^\star,\blambda^\star)=0$.
\item[(I.1)] If $c\geq \max\{\kappa_N/\mu_\Phi,\kappa_N/\mu_{\Psi}\}$,
\begin{align}\label{eq: bound lower lyapunov 2}
V(\bx,\blambda)\geq \frac{1}{2}(\mu_{\Phi}c-\kappa_N)||\bx-\bx^\star||^2 + \frac{1}{2}(\mu_{\Psi}c-\kappa_N)||\blambda-\blambda^\star||^2\geq 0.
\end{align}
\end{description}
If in addition $f$ is $\mu_f$-strongly convex relative to $\Phi$ then:
\begin{description}
\item[(II.2)] For any $c\geq \max\{((\kappa_N-\mu_fL_\Phi)/\mu_\Phi,\kappa_N/\mu_{\Psi}\}$,
\begin{equation}\label{eq: another bound lower lyapunov 2}
V(\bx,\blambda)\geq \frac{1}{2}(\mu_{\Phi}c+\mu_fL_{\Phi}-\kappa_N)||\bx-\bx^\star||^2 + \frac{1}{2}(\mu_{\Psi}c-\kappa_N)||\blambda-\blambda^\star||^2\geq 0.
\end{equation}
\end{description}
\begin{description}
\item[(III)] Let $\hat\mu = \min\{\mu_\Phi,\mu_\Psi\}$. Then, 
\begin{equation}\label{eq: bound upper lyapunov 2}
V(\bx,\blambda)\leq \left(c+\frac{3\kappa_N+(1+\hat\mu)\alpha(\Phi)}{\hat\mu}
\right)\left( \sum_{i=1}^N D_\Phi(x^\star,x^i)+\sum_{i=1}^ND_\Psi(\lambda^\star,\lambda^i) \right).
\end{equation}
\end{description}
\end{lemma}
\begin{proof}
\color{black}
Property (I) is obvious. For (I.1) we bound $V_1, \ V_2$ using the strong convexity of $\Phi$ and $\Psi$ respectively:
\begin{equation}\label{eq:bound v1v2}
\begin{split}
V_1(\bx)&=\sum_{i=1}^N D_\Phi(x^\star,x^i)\geq \frac{\mu_\Phi}{2}\|\bx^\star-\bx\|^2, \\
V_2(\blambda)&=\sum_{i=1}^N D_\Psi(\lambda^\star,\lambda^i)\geq \frac{\mu_\Psi}{2}\|\blambda^\star-\blambda\|^2.
\end{split}
\end{equation}
We note that the convexity of $f$ implies that $D_f(\bx,\bx^\star)\geq 0$, and therefore we can bound $V_3$ as follows,
\begin{align}
V_3(\bx,\blambda)&\geq\inp{\bx-\bx^\star}{\cL(\blambda-\blambda^\star)}
\geq -\frac12(\|\cL^\frac12(\bx-\bx^\star)\|^2_2+\|\cL^\frac12(\blambda-\blambda^\star)\|^2_2)\\
&\geq -\frac{\kappa_N}{2}(\|\bx-\bx^\star\|^2+ \|\blambda-\blambda^\star\|^2),
\end{align}
where in the second inequality we used \eqref{eq: bound inp} and in the third one \eqref{eq: Rayleigh quotient l2}.
Using the bound for $V_3$ above and the two bounds in \eqref{eq:bound v1v2} we obtain (I.1).

If, in addition, $f$ is strongly convex relative to $\Phi$ then,
\begin{align}
\sum_{i=1}^N D_f(x^i,x^\star)\geq \mu_f\sum_{i=1}^N D_\Phi(x^i,x^\star)
\geq\frac{\mu_f\mu_\Phi}{2}\|\bx^\star-\bx\|^2.
\end{align}
Using the preceding inequality to bound $V_3$ along with \eqref{eq:bound v1v2} we obtain the (II.2).

For the upper bound in (III) we bound the the first term in $V_3$ using the symmetry bound in \eqref{eq:symmetry_Bregman},
\begin{equation}\label{lyapunov:ub1}
\sum_{i=1}^N D_f(x_i,x^\star)\leq \alpha(\Phi)\sum_{i=1}^N D_\Phi(x^\star,x_i).
\end{equation} 
For the second term in $V_3$ we use \eqref{eq: bound inp} again and for any $\gamma>0$,
\begin{equation}\label{lyapunov:ub2}
\begin{split}
\inp{\bx-\bx^\star}{\cL(\blambda-\blambda^\star)}
&\leq \frac {\kappa_N}{2}(\gamma\|\bx-\bx^\star\|^2
+\frac1\gamma\|\blambda-\blambda^\star\|^2)\\
& \leq \frac{\kappa_N\gamma}{\mu_\Phi}
\sum_{i=1}^N D_\Phi(x^\star,x^i)+\frac{\kappa_N}{\gamma\mu_\Psi}\sum_{i=1}^N D_\Psi(\lambda^\star,\lambda^i) \\
&\leq \frac{\kappa_N}{\hat\mu}\left(\gamma\sum_{i=1}^N D_\Phi(x^\star,x^i)+\frac{1}{\gamma}\sum_{i=1}^N D_\Psi(\lambda^\star,\lambda^i) \right)\\
&\leq \frac{\kappa_N}{\hat\mu}\sum_{i=1}^N D_\Phi(x^\star,x^i)
+\frac{\kappa_N+\alpha(\Phi)}{\hat\mu}\sum_{i=1}^N D_\Psi(\lambda^\star,\lambda^i) ,
\end{split}
\end{equation}
where in the second inequality we used the relative strong convexity of $\Phi$ and for the last inequality we set 
$\gamma=\frac{\kappa_N}{\kappa_N+\alpha(\Phi)} \leq 1.$
Finally, for the last term in $V_3$ we use the bound from \eqref{eq: Rayleigh quotient l2}, the strong convexity of $\Phi$ and the definition of $\hat\mu$,
\begin{equation}\label{lyapunov:ub3}
\|\cL^\frac12\bx\|^2_2=\|\cL^\frac12(\bx-\bx^\star)\|^2_2\leq \kappa_N \|\bx-\bx^\star\|^2\leq \frac{2\kappa_N}{\hat\mu} \sum_{i=1}^N D_\Phi(x^\star,x^i)
\end{equation} 
Using the upper bounds for the three terms in $V_3$ we obtain the bound in \eqref{eq: bound upper lyapunov 2}.
\end{proof}
Below we state and prove the main result of this section.

\begin{proposition}[Convergence of the preconditioned dynamics in \eqref{eq:alm continious sde prec}] \label{prop:expconv_stoch_dual}\color{black}
Let Assumptions \ref{ass:f}-\ref{ass:mirror1} hold and suppose that $\kappa_g>0$ (see \eqref{eq: General Rayleigh quotient}. Consider the dynamics in \eqref{eq:alm continious sde prec}. Let the Lyapunov function $V_t$ be defined as in \eqref{eq:laypunov exact primal v2}, then
\begin{equation}\label{eq: convergace rate edmd}
\mathbb E[V_T]\leq
e^{-rT}\mathbb E[V_0]+\frac{\sigma^2}{2}\mathbb E\left[\int_0^T e^{-r(T-s)} M_s ds\right]
\end{equation}
where $V_T := V(\mathbf{x}_T,\blambda_T)$, 
\begin{align}\label{eq:rate}
r=\frac{2\kappa_g}{c\hat\mu+(1+\hat\mu)\alpha(\Phi)+3\kappa_N}>0,
\end{align}
\[
M_s=
 c\textnormal{tr}(\nabla^2\Phi^*(\bz_s))
+\left(\Delta\cdot\nabla\Phi^*(\bz_s)+\textnormal{tr}(\nabla^2\Phi^*(\bz_s)\nabla_{xx}^2 L(\bx_s,\blambda_s)\nabla^2\Phi^*(\bz_s))\right)
\]
where $\hat\mu=\min\{\mu_\Phi,\mu_\Psi\}$ and $c\geq \max\{2\kappa_{\beta,N}\mu_{\Psi},\kappa_N/\mu_\Phi,\kappa_N/\mu_\Psi\},$ and the operator $\Delta\cdot\nabla\Phi^*(\bz)$
is given by  $[\Delta\cdot \nabla \Phi^*(z^i)]_j
=\sum\limits_{k=1}\limits^{d} \partial^2_{kk}\partial_j \Phi^*(z^i), \  i=1,\ldots,N, \ j=1,\ldots,d$.

\end{proposition}

\begin{proof}\color{black}
Since $x^i=\nabla \Phi^*(z^i)$ it follows from It\^o's Lemma that,
\begin{align}
dx^i_t=\nabla^2 \Phi^*(z^i)dz^i_t+\frac12 \sigma^2 \Delta\cdot \nabla \Phi^*(z^i)dt,
\end{align}
where the $j^{th}$ element of the It\^o correction term is 
$[\Delta\cdot \nabla \Phi^*(z^i)]_j
=\sum\limits_{k=1}\limits^{d} \partial^2_{kk}\partial_j \Phi^*(z^i)$. For ease of exposition we make the following definition,
\[
M_1(t)=\frac{\sigma^2}{2}\int_0^t \textnormal{tr}(\nabla^2\Phi^*(z_s))ds-\sigma\int_0^t \inp{\nabla\Phi^*(\bz_s)-\nabla\Phi^\star(\bz^\star)}{dB_s}.
\]
Next note that $V_1$ defined in \eqref{eq:lyapunov v1_3} can be written as follows $V_1(\bx)=\sum_{i=1}^N D_\Phi^{*}(x^i,x^\star)$ and therefore,
\begin{equation}\label{eq: bound V1}
\begin{split}
    dV_1(t) &= -\inp{\nabla\Phi(\bz_t)-\nabla\Phi^\star(\bz^\star)}{\nabla f(x_t)+\cL\bx_t+\cL\blambda_t}dt+dM_1(t)\\
&= -(\inp{\bx_t-\bx^\star}{\nabla f(\bx_t)-\nabla f(\bx^\star)}+\inp{\bx_t-\bx^\star}{\nabla f(\bx^\star)+\cL\bx_t+\cL\blambda_t})dt+dM_1(t)\\
&\leq -\inp{\bx_t-\bx^\star}{\cL\bx_t+\cL(\blambda_t-\blambda^\star)}dt+dM_1(t),
\end{split}    
\end{equation}
where for the last inequality we used the convexity of $f$ and the optimality condition $\nabla f(\bx^\star)=-\cL\blambda^\star$. 
We also re-write the second equation in \eqref{eq:alm continious sde prec} in terms of $\lambda_t$,
\[
d\blambda_t=\nabla^2\Psi(\blambda_t)^{-1}\cL\bx_t
\]
Using the equation above we obtain the following expression for $V_2$ in \eqref{eq:lyapunov v1_3}, 
\begin{align}
    dV_2(t) &= \langle -\nabla^2\Psi(\blambda_t)(\blambda^\star-\blambda_t),\nabla^2\Psi(\blambda_t)^{-1}\mathcal{L}\mathbf{x}_t \rangle dt\\
    &= (\blambda_t-\blambda^\star)^T \mathcal{L} (\mathbf{x}_t-\mathbf{x}^\star)dt.
\end{align}
Adding the bounds for $V_1$ and $V_2$ we obtain,
\begin{equation}\label{eq: bound ito cv1v2}
\begin{split}
d(V_1(t)+V_2(t))&\leq
-\inp{\bx_t-\bx^\star}{\cL(\bx_t-\bx^\star)}dt + dM_1(t)
\\
&\leq -\frac{1}{\kappa_{\beta,N}}\|\cL(\bx_t-\bx^\star)\|_2^2 dt
+dM_1(t)
\end{split}
\end{equation}
where we used Lemma \ref{lemma: bound laplacian} to bound the first term. Using the definition of the Augmented Lagrangian in \eqref{eq:al difinition} we can re-write $V_3$ from 
\eqref{eq:lyapunov v1_3} as follows,
\[
V_3(\bx,\blambda)=L(\bx,\blambda)-f(\bx^\star).
\]
Therefore,
\begin{equation}\label{eq: dv3 term}
\begin{split}
dV_3(t)&=-\|\nabla_x L(\bx,\blambda)\|_{\nabla^2{\Phi^*(\bz)}}dt+\|\nabla_\lambda L(\bx,\blambda)\|_{\nabla^2{\Psi^*(\bmu)}}dt+dM_2(t),\\
& \leq -\frac{2\kappa_g}{\hat\mu}\left( D_{\Phi}(\bx^\star,\bx_t)+D_{\Psi}(\blambda^\star,\blambda_t)\right)dt
+\frac{2}{\mu_\Psi} \|\cL(\bx-\bx^\star)\|^2dt+dM_2(t)
\end{split}
\end{equation}
where the last inequality follows from \ref{lemma:bound on lagrangian and mirror map v1} and $M_2(t)$ is defined as follows,
\[
\begin{split}
M_2(t)&=\frac{\sigma^2}{2}\int_0^t \left(\Delta\cdot\nabla\Phi^*(\bz_s)+\textnormal{tr}(\nabla^2\Phi^*(\bz_s)\nabla_{xx}^2 L(\bx_s,\blambda_s)\nabla^2\Phi^*(\bz_s))\right)ds
\\
&+\sigma\int_0^t \inp{\nabla L(\bx_s,\blambda_s)}{\nabla^2\Phi^*(\bz_s)d\mathbf{B}_s} 
\end{split}
\]
Using the bounds in \eqref{eq: bound V1} and  \eqref{eq: bound ito cv1v2} we obtain,
\[
\begin{split}
dV_t&\leq -\frac{2\kappa_g}{\hat\mu}\left( D_{\Phi}(\bx^\star,\bx_t)+D_{\Psi}(\blambda^\star,\blambda_t)\right)dt
 +\left(\frac{2}{\mu_\Psi}-\frac{c}{\kappa_{\beta,N}} \right)\|\cL(\bx-\bx^\star)\|^2dt+cdM_1(t)+dM_2(t) \\
& \leq  -\frac{2\kappa_g}{\hat\mu} \left(c+\frac{3\kappa_N+(1+\hat\mu)\alpha(\Phi)}{\hat\mu}
\right)^{-1}V(\bx_t,\blambda_t)dt+ cdM_1(t)+dM_2(t)\\
&=-rV_t dt+cdM_1(t)+dM_2(t)
\end{split}
\]
where in the second inequality we used \eqref{eq: bound upper lyapunov 2} in Lemma \ref{eq:laypunov exact primal v2} and the condition $c\geq 2\kappa_{\beta,N}/\mu_\Psi$, and $r$ is defined in \eqref{eq:rate}. Taking expectations and applying Gronwall's Lemma we obtain the final result. 
\end{proof}
{\color{black}
\begin{remark} \textit{Convergence to an invariant distribution:}
    Proposition \ref{prop:expconv_stoch_dual} can be extended to establish a standard Foster-Lyapunov condition, $\mathcal{A}V_t\leq rV_t+\overline d$, where $\mathcal{A}$ is the infinitesimal generator of \eqref{eq:alm continious sde prec}. By imposing some additional (and not particularly restrictive) assumptions on $\Phi,f$ such that 
    $c\textnormal{tr}(\nabla^2\Phi^*(z_s))+\left(\Delta\cdot\nabla\Phi^*(\bz_s)+\textnormal{tr}(\nabla^2\Phi^*(\bz_s)\nabla_{xx}^2 L(\bx_s,\blambda_s)\nabla^2\Phi^*(\bz_s))\right)<\overline d$ holds, one then
     establishes the existence of an invariant distribution, \cite[Theorem 5.3]{meyn1993stability}. One could go further to establish exponential convergence to equilibrium via \cite[Theorem 6.1]{meyn1993stability}, but this would require establishing a minorization condition for irreducibility or equivalently showing controllability of the dynamics in \eqref{eq:alm continious sde prec} when the noise is replaced with a deterministic control input. This requires further assumptions on $\Phi$ and for the sake of brevity details are omitted.
\end{remark}
}
\revision{
The result in Proposition \ref{prop:expconv_stoch_dual} is quite general since it allows \emph{any} mirror map (within assumptions) in the primal and/or dual.
In the standard decentralized case we can only establish the superiority of the mirror descent method over gradient methods for specific cases.
For example the case where $\Phi$ is the negative entropy and $\mathcal X$ is the simplex, then mirror descent is known to be more efficient both in theory and in practice (see \cite[Example 9.10]{beck17book}).
Therefore in order to generate similar insights into the distributed case it is necessary to investigate some specific choices of mirror maps.
In the next section we motivate a particular preconditioner that relies on convex duality theory.
}
\color{black}
\subsection{Choosing the Network Preconditioner}\label{sec:choose precondtioner} 
For many interesting applications good mirror maps are known, and the advantages of mirror descent over gradient descent are well understood. 
Unfortunately, it is not clear how to select a good mirror map for the space of Lagrange multipliers. 
\revision{
However, if we restrict the mirror map to be quadratic, 
then it is possible to identify one possible mirror map that works well in practice (see also the numerical results in \ref{sec:num}). 
In \ref{sec:Discretization Analysis} (see \ref{remark:convergence quadratic}) we show that if we choose the mirror map to be the Hessian of the dual of \eqref{eq:distoptobj0} and if the objective function is quadratic, then the algorithm would converge in a single iteration (when $\sigma=0$). This observation is reminiscent of the Newton method for unconstrained problems and so we propose that using the Hessian of the Lagrangian is a good choice for $\Psi$. Of course, computing the Hessian of the Lagrangian could be expensive but we argue that the extra computation associated with approximating the Hessian of the dual function could be justified in the scenario where the Laplacian matrix is ill-conditioned. 
We can back this claim exactly using the example in \ref{remark:convergence quadratic} and also using more realistic problems in \ref{sec:num}. We also note that our discretization scheme does not rely on the exact computation of the Hessian of the dual.} 

Below we briefly outline the derivation of the Hessian for the dual function in the deterministic setting and when $f$ is strongly convex. 
For simplicity we only consider the Lagrangian $f(\bx)+\inp{\blambda}{\cL\bx}$ (and not the Augmented Lagrangian). To extend for the Augmented Lagrangian the extra term $\frac12\|\cL^\frac12\bx\|_2^2$ can easily be added in the expressions below but for simplicity is omitted. 
In particular, the (negative) dual function $q(\blambda)$ is defined as follows,
\begin{equation}\label{eq:dual function}
q(\blambda)=\max_{\bx\in\mathcal{X}^N} \left\{-(f(\bx)+\inp{\blambda}{\cL \bx }) \right\},
\end{equation} 
Let $\bx(\blambda)$ be the maximizer of \eqref{eq:dual function}. We know that for the strongly convex case, $\bx(\blambda)$ is unique and the gradient of the dual function is given by ( \cite[Theorem 6.3.3]{bazaraa2013nonlinear}),
\[
\nabla_{\blambda} q(\blambda)=-\cL \bx(\blambda).
\]
Applying the KKT conditions to \eqref{eq:dual function} and differentiating with respect to $\lambda$ we also have that,
\[
\nabla^2 f(\bx(\blambda))\frac{d\bx(\blambda)}{d\blambda}+\cL=0.
\]
Using the preceding equation we obtain the following expression for the Hessian of the dual function,
\[
\nabla^2 q(\blambda)=-\cL\frac{d\bx(\blambda)}{d\blambda}=
\cL(\nabla^2 f(\bx(\blambda)))^{-1}\cL.
\]
Unfortunately, even if $f$ is strongly convex, the dual function in \eqref{eq:dual function} is only convex (and not strongly convex). So we cannot directly use the Hessian of the dual function above, and instead we proposed to use,
\[
\nabla^2\Psi(\blambda)=\cL_\beta (\nabla^2 f(\bx(\blambda)))^{-1}\cL_\beta,
\]
where $\cL_\beta$ is the $\beta$-regularized Laplacian defined in \eqref{eq:reglapl}.
Note that we can avoid inverting the Hessian of $f$ at every iteration and instead work with the convex conjugate of $\Psi$ which occurs a one-off cost of diagonalizing the regularized Laplacian,
\[
\nabla^2\Psi^*(\blambda)=\cL_\beta^{-1} (\nabla^2 f(\bx(\blambda)))\cL_\beta^{-1}.
\]
\revision{
A major computational challenge associated with the derivation above is the computation of $\bx(\blambda)$.
Unless $f$ is a particularly simple function, it is not possible to solve \eqref{eq:dual function} in closed form.
In the next section we will address this issue using a discretization approach that is similar to a Gauss-Seidel scheme. 
}

\section{Discretization Analysis}\label{sec:Discretization Analysis}\color{black}

In the convergence analysis of Section \ref{sec: convergence}, any mirror map that is compatible with Assumption \ref{ass:mirror1} can be used. 
However, this flexibility means that the dynamics in \eqref{eq:alm continious sde prec} have the potential to lead to good or bad algorithms depending upon the specific mirror map selection. 
As a result the convergence rates we obtained in the previous section are not necessarily the most tight.
To address this issue we examine the algorithm once the dynamics in \eqref{eq:alm continious sde prec} are discretized and we specifically focus on quadratic mirror maps. 
For the remainder we will restrict to \begin{align}
    \Phi(\bx)=\frac{1}{2}\bx^\top Q \bx, \quad   \Psi(\blambda)=\frac{1}{2}\blambda^\top R \blambda,
\end{align}
where $Q,R$ are positive-definite matrices. Nevertheless, the insights of the continuous time analysis obtained previously is crucial since the results of this section rely on the same Lyapunov function from Section 
\ref{sec: convergence}. 
By restricting the analysis to the case of quadratic mirror maps we will be able to identify mirror maps for both the primal and Lagrangian dual variables that lead to favorable condition numbers. 
Like all preconditioning methods, a balance needs to be struck between the computational cost of computing the preconditioner and the algorithm's convergence rate. One of the aims of the analysis in this section is to provide a detailed description of this inherent trade-off (see \ref{fig:discrete constnants} and the discussion in the last paragraph of this section).

We propose to apply an explicit Gauss-Seidel discretization scheme to \eqref{eq:alm continious sde prec},
\[
\begin{split}\label{eq: gauss-seidel zm}
\bz_k&=\bz_{k-1}-\delta(\nabla f(\bx_{k-1})+\cL\bx_{k-1}
+\mathcal L\blambda_{k-1})+\sqrt{\delta}\sigma\ \bxi_k, \\
\bmu_k&=\bmu_{k-1}+\delta\mathcal L\bx_{k},
\end{split}
\]
where $\bxi_{k}\sim N(0,I_{Nd})$. Note that we use the notation $\bx_k$ to denote the approximation to $\bx(k\delta)$, with similar notation for the rest of the variables. 
In this work, our emphasis primarily lies on employing this scheme for optimization rather than for sampling purposes. Consequently, the convergence of $\bx_k$ to $\bx(k\delta)$ or the density of $\bx_k$ are outside the scope of this paper. Note that if we were to follow the motivation discussed in \ref{sec:choose precondtioner} then we would have $\bx_k$ to be the solution of the dual problem. Instead, in our scheme below we just use a simple Gauss-Seidel scheme. The two schemes (i.e. Gauss-Seidel and exact solution of the dual) would coincide if the solution of the dual problem could be obtained via a single mirror-descent iteration. 

Let $(\bz^\star,\bmu^\star)$ denote the duals of the optimal solution defined in \eqref{eq:define optimal pair}. We introduce the following notation,
\[
\begin{split}
\hat \bz_k&=\bz_k-\bz^\star, \\
\hat \bmu_k&=\bmu_k-\bmu^\star,\\
\nabla \hat f_k&=\nabla f(\bx_k)+\cL\bx_k-(\nabla  f(\bx^\star)+\cL \bx^\star ).
\end{split}
\]
Subtracting $\bz^\star$ and $\bmu^\star$) from both sides of the first (second) equation in \eqref{eq: gauss-seidel zm}, using 
$\hat \bx_k=Q^{-1}\hat \bz_k$, $\hat \blambda_k=R^{-1}\hat \bmu_k$
and the KKT optimality conditions we obtain,
\begin{equation}\label{alg:discrete}
\begin{split}
\hat\bx_k&=\hat\bx_{k-1}-\delta Q^{-1}(
\nabla \hat f(\bx_{k-1})
+\mathcal L\hat\blambda_{k-1})+\sqrt{\delta}\sigma Q^{-1} \bxi_k, \\
\hat \blambda_k&=\hat \blambda_{k-1}+\delta R^{-1}\mathcal L\hat\bx_k.\end{split}
\end{equation}

\begin{remark}
 The assumption that the mirror maps are quadratic is not as restrictive as it might first appear.
Firstly, the consensus constraint is an equality, therefore the Lagrangian dual variables are vectors in $\R^{Nd}$. 
Therefore using a $Q$-norm for some positive definite matrix $Q$ is the natural choice.
For the Lagrangian duals any other commonly used mirror map such as the negative entropy would not lead to convergent algorithms. 
The domain of the negative entropy is the positive orthant and it would therefore violate our assumptions. 
Secondly, regarding the mirror map of the primal variables, we can extend our setting in the discretization scheme using local linearization and a first order Taylor expansion and set $Q_k=\nabla^2\Phi(\bx_{k-1})$, i.e. discretize the right-hand side of \eqref{eq:alm continious sde prec} using
\[
\nabla \Phi(\bx_k)-\nabla \Phi(\bx_{k-1})= \nabla^2\Phi(\bx_{k-1})(\bx_k-\bx_{k-1})+
O(\|\bx_k-\bx_{k-1}\|^2),
\]
and similarly for $R_k$ and $\blambda_k$. Our analysis below is for the case where $Q$ and $R$ are the same at each iteration but similar arguments can be used to extend it to the non homogeneous case. 
\end{remark}

We begin by specializing our assumptions on relatively strong convexity and relative smoothness to the case of quadratic maps.  
\begin{lemma}\label{lemma:quad map inqualities}
Suppose that $\Phi(\bx)=\frac12\bx^\top Q\bx$, where $Q$ is a positive definite matrix, then: 
\begin{description}
    \item[(i)] If $f$ is $L_f$-smooth relative to $\Phi$ then,
\begin{equation}\label{eq: smooth quad case}
\|\nabla f(\by)-\nabla f(\bx)+\cL(\by-\bx) \|^2_{Q^{-1}}
 \leq L'_f \inp{\nabla f(\bx)-\nabla f(\by)+\cL(\bx-\by)}{\bx-\by}
\end{equation}
where $L'_f=L_f+\kappa_N/\mu_\Phi$.
    \item[(ii)] If $f$ is $\mu_f$-strongly convex relative to $\Phi$ we have,
    \begin{equation}\label{eq: sc quad case}
\inp{\nabla f(\bx)-\nabla f(\by)+\cL(\bx-\by)}{\bx-\by}\geq \mu_f\|\bx-\by\|^2_{Q}
\end{equation}
\end{description}
\end{lemma}
\begin{proof}
(i) Consider the function $g(\by)=f(\by)+\|\cL^{\frac12}\by\|^2-\inp{\nabla f(\bx)+\cL \bx}{\by}$ and note that it is minimized by $\bx$. Moreover, since $f$ is relatively smooth w.r.t $\Phi$ implies that for any $\by'$,
\[
\nabla^2 f(\by')+\cL\preceq L_fQ+\kappa_NI
\preceq (L_f+\frac{\kappa_N}{\mu_\Phi})Q,
\]
where we used Assumption \ref{ass:mirror1} and $\kappa_N$ was defined in \eqref{eq: Rayleigh quotient l2}. Therefore $g$ is also relatively smooth w.r.t $\Phi$:
\[
D_g(\by',\by)\leq L'_f D_{\Phi}(\by',\by)=\frac{L'_f}{2}\|\by'-\by\|^2_Q
\]
Using the equation above at with $\by'=\by-\frac{Q^{-1}}{L'_f}(\nabla f(\by)+\cL \by)$
\[
\begin{split}
g(\bx)&\leq g(\by-\frac{1}{L'_f}Q^{-1}(\nabla f(\by)+\cL \by))
\leq g(\by)-\frac{1}{2L'_f}\|\nabla f(\by)-\nabla f(\bx)+\cL(\bx-\by) 
\|^2_{Q^{-1}}
\end{split}
\]
Where the first inequality follows from the fact that $\bx$ minimizes $g$. 
Using the definition of $g$ on the preceding equation we obtain,
\begin{align}
\inp{\nabla f(\bx)+\cL \bx}{\bx-\by}-(f(\bx)+\|\cL^{\frac12}\bx\|^2)
\geq& -(f(\by)+\|\cL^{\frac12}\by\|^2) \\
&
+\frac{1}{2L'_f}\|\nabla f(\by)-\nabla f(\bx)+\cL(\bx-\by) \|^2_{Q^{-1}}.
\end{align}
Finally, using the inequality above with the role of $\bx$ and $\by$ reversed and adding the two inequalities we obtain \eqref{eq: smooth quad case}.

(ii) For the second part we use the mean-value theorem to obtain,
\[
\nabla f(\bx)-\nabla f(\by)+\cL(\bx-\by)
=(\nabla^2 f(\bx')+\cL)(\bx-\by)
\]
for some $\bx'=\tau\bx+(1-\tau)\by$ with $0<\tau<1$. Multiplying both sides by $\bx-\by$ and using the assumption that $f$ is $\mu_f$-strongly convex relative to $\Phi$ we obtain,
\[
\begin{split}
\inp{\nabla f(\bx)-\nabla f(\by)+\cL(\bx-\by)}{\bx-\by}
&=\inp{(\bx-\by)(\nabla^2 f(\bx')+\cL)}{\bx-\by} \\
&\geq \mu_f\inp{\bx-\by}{Q (\bx-\by)}=\mu_f\|\bx-\by\|^2_{Q},
\end{split}
\] 
where for the inequality we used the fact that $\cL$ is positive semi-definite.
\end{proof}

Next we introduce two generalized Rayleigh quotients that will play a key role in our analysis and are inspired by the definitions in \eqref{eq: Rayleigh quotient l2} and 
\eqref{eq: General Rayleigh quotient}. However, because the mirror maps in this section are assumed to be quadratic (and therefore symmetric) we will be able to obtain more refined estimates with the definitions below:
\begin{equation}\label{eq:sigma upper}
\overline{\sigma}(\mathcal L)^2=\sup_{\bx\in\mathcal{X}^N} \frac{\inp{\mathcal L\bx}{R^{-1}\mathcal L \bx}}{\|\bx\|^2_{Q}},
\end{equation}
\begin{equation}\label{eq:sigma lower}
\underline{\sigma}(\mathcal L)^2=\inf_{\blambda\in\range(\mathcal L)} \frac{\inp{\mathcal L\blambda}{Q^{-1}\mathcal L \blambda}}{\|\blambda\|^2_{R}}. 
\end{equation}
In the standard setting where $Q=R=I$ (i.e. when our algorithm reduces to an exact distributed gradient descent algorithm) then $\overline{\sigma}(\mathcal L)$ is the largest eigenvalue value of $\mathcal L$ and $\underline{\sigma}(\mathcal L)>0$ the second smallest. 
However, in the framework we analyze in this paper we have the opportunity to design the algorithm by choosing the $\Phi$ and $\Psi$. 
To see the utility of the definitions above suppose that $\cL$ was invertible; (of course $\cL$ is not invertible, but for the sake of argument suppose that it was). Then if $R=\mathcal LQ^{-1}\mathcal L$ and  then 
$\underline{\sigma}(\mathcal L)=\overline{\sigma}(\mathcal L)=1$. 
Therefore, the condition number of the graph Laplacian will play no role in the convergence rate of the algorithm. 
While not possible to invert $\mathcal L$, we can achieve a similar effect by using the $\beta-$regularized Laplacian (see Remark \ref{remark:convergence quadratic} below). 
The definition of the two Rayleigh quotients above will be useful in the analysis, but the Lemma below is useful for computing these constants using standard methods.  

\begin{lemma}\label{lemma:eigs sigma bar}
\begin{description}
\item[(i)]The second smallest eigenvalue value of $R^{-\frac12}\cL Q^{-1}\cL R^{-\frac12}$ is equal to $\underline{\sigma}(\mathcal L)$. 
\item[(ii)] The largest eigenvalue value of $Q^{-\frac12}\cL R^{-1}\cL Q^{-\frac12}$ is equal to $\overline{\sigma}(\mathcal L)$.
\end{description}
\end{lemma}
\begin{proof}
(i) follows from the  change of variables $\mathbf{\kappa}=R^{-\frac12}\blambda$ and noting that the range of $\blambda\in\mathcal{L}$ excludes the possibility of $\mathcal L \blambda=0$ which would otherwise correspond to smallest eigenvalue. (ii) follows using $\bx=Q^{-\frac12}\by$.
\end{proof}
In order to be able to use the lower bound in \eqref{eq:sigma lower} we need to ensure that the $\blambda_k$ generated by the algorithm satisfies the condition 
$\blambda_k\in\range(\mathcal L)$. 
The Lemma below shows that this is easy to ensure as long as the algorithm is initialized correctly. 
\begin{lemma}\label{lemma: range space lambda}
Suppose that \eqref{alg:discrete} is initialized at $\blambda_0=0$, $\bx_0\in\mathcal{X}^N$ and $R$ has full rank. Then $\blambda_k\in\range(\mathcal L)$, $\forall k>0$.
\end{lemma}
\begin{proof}
We proceed by induction. For $\blambda_0=0$ then $\blambda_1=R^{-1}\cL\bx_1$ therefore $\blambda_1\in\range(R^{-1}\cL)$. 
Since $R^{-1}$ has full rank it follows from \cite[Corollary 4.2.12]{meyer2023matrix} that $\range(R^{-1}\cL)=\range(\cL)$.
Now suppose that $\blambda_{m-1}\in\range(\cL)$, then
\[
\blambda_{m}=\blambda_{m-1}+\delta R^{-1}\cL \bx_m=\delta R^{-1}\cL\left(\sum_{i=1}^{m-1}\bx_i+(\bx_m-\bx^\star)\right),
\]
and therefore $\blambda_{m}\in\range(\cL)$.
\end{proof}
\begin{example}\label{remark:convergence quadratic}
It is instructive to examine the case where the conditioning of the optimization model is solely due to the condition number of the graph. Suppose $\mathcal{X}=\mathbb{R}$ and $f(\bx)=\frac{\alpha}{2}\|\bx-\bf{b}\|^2$, where $\alpha>0$, and $\bx$, and $\bf{b}$ are in $\R^N$. In this case minimizing $f(\bx)$ subject to the constraint $\cL\bx=0$ has the unique solution $\frac{1}{N}\sum_{i=1}^N b_i$. Let $C=\{\bx\in\R^N~|~\cL\bx=0\}$ and let $C^{\perp}$ be its orthogonal complement. We also denote the unit vector as $e=\mathbf 1_N$ and $\Pi_{C^\perp}$ as the orthogonal projection on $C$ so that
\[
\begin{split}
\Pi_{C}&=\frac{1}{N}ee^\top \\
\Pi_{C^\perp}&=I-\frac{1}{N}ee^\top
\end{split}
\]
It follows from the above that $\cL{\Pi_{C}}=0$ and $\cL\Pi_{C^\perp}=\cL$. We will use these properties of the orthogonal projections on $C$ below. Next, suppose that $Q=\alpha I$ and $R=\cL_\beta Q^{-1}\cL_\beta$ then using Lemma \ref{lemma:eigs sigma bar} we find $\overline{\sigma}(\mathcal L)=\underline{\sigma}(\mathcal L)=1$.
If $\sigma=0$, $\blambda_0=0$ and $\delta=1$ in \eqref{alg:discrete} then the algorithm converges in one iteration. 
To see this note that,
 \[
 \begin{split}
 \bx_1&=\bx_0-(\alpha I+\cL)^{-1}((\alpha I+\cL)\bx_0-\alpha\bf{b}))=(\alpha I+\cL)^{-1}
 (\alpha\bf{b}),\\
 \blambda_1&=\cL_\beta^{-1}(\alpha I+\cL)\cL_\beta^{-1}\cL\bx_1.
 \end{split}
 \]
Starting at an arbitrary $\bx_0$, after some straightforward algebra we obtain,
\[
\begin{split}
\bx_2&=(\alpha I+\cL)^{-1}(\alpha \bf{b}-\cL\blambda_1)\\
 &=(\alpha I+\cL)^{-1}(\alpha\bf{b}-\cL\cL_\beta^{-1}(\alpha I+\cL)\cL_\beta^{-1}
 \cL(\alpha I+\cL)^{-1}(\alpha\bf{b}))\\
 &=(\alpha I+\cL)^{-1}(\alpha{\bf b}- \Pi_{C^\perp}(\alpha I+\cL)
  \Pi_{C^\perp}(\alpha I+\cL)^{-1}(\alpha{\bf b}))\\
&=(\alpha I+\cL)^{-1}(I-\Pi_{C^\perp})(\alpha{\bf b})=(\alpha I+\cL)^{-1}\Pi_{C}(\alpha{\bf b}),
 \end{split}
\]
where we used the facts that orthogonal projections are idempotent and the fact that $\cL{\Pi_{^\perp}}=\cL$. Finally we use the fact that,
\[
\begin{split}
(\alpha I+\cL)^{-1}={\alpha}^{-1}( I+{\alpha}^{-1}\cL)^{-1}
&={\alpha}^{-1}( I+{\alpha}^{-1}\cL)^{-1}(I+{\alpha}^{-1}\cL-{\alpha}^{-1}\cL)\\
&={\alpha}^{-1}I-{\alpha}^{-2}(I+{\alpha}^{-1}\cL)^{-1}\cL.
\end{split}
\]
Therefore,
\[
\bx_2={\alpha}^{-1}I-{\alpha}^{-2}(I+{\alpha}^{-1}\cL)^{-1}\cL\Pi_{C}(\alpha{\bf b}),
\]
and since $\cL\Pi_{C}(\alpha{\bf b})=0$ it follows that,
\[
\bx_2=\Pi_{C}({\bf b})=\frac1N\sum_{i=1}^N b_i.
\]
Note that if we chose $Q=R=I$ then the convergence (and step-size $\delta$) of the algorithm would depend on the condition number of the graph.
\end{example}
We are now ready to state the main result of this section.
\begin{proposition}\label{prop: discretization main}
Suppose that,
\begin{equation}\label{eq:stepsize condition}
\delta<\min\left\{\frac{1}{2L'_f},\frac{\mu_f}{\underline\sigma(\cL)^2}\right\}.
\end{equation}
Then, 
\begin{equation}\label{eq:convergance discrete}
\mathbb E[ V(\bx_k,\blambda_k)]
\leq \rho^{k} \mathbb E [ V(\bx_{0},\blambda_{0})]
+\frac{\delta\sigma^2\textnormal{tr}(Q^{-1})}{2(1-\rho)}.
\end{equation}
where,
\begin{equation}\label{eq:lyapunov discrete}
V(\bx_k,\blambda_k)=(1-\delta^2\overline\sigma(\cL)^2)D_{\Phi}(\bx_k,\bx^\star)+D_{\Psi}(\blambda_k,\blambda^\star),
\end{equation}
and $0<\rho=\max\{\rho_x,\rho_\lambda\}<1$, with $\rho_x=1+\delta^2\mu_fL'_f-\delta\mu_f$, $\rho_\lambda=1-\delta^2\underline\sigma(\cL)^2$.
\end{proposition}
\begin{proof}
To simplify notation we will use $I(\xi_k)$ to denote a random variable with zero expectation and we will omit the exact terms when they will not play a role in the estimates. 
We expand the first term in \eqref{eq:lyapunov discrete} as follows,
\[
\begin{split}
D_\Phi(\bx_k,\bx^\star)
&=\frac{1}{2}\|\hat\bx_k\|_{Q}^2 \\
&=\frac{1}{2}\|Q^{\frac12}\hat\bx_{k-1}-
\delta Q^{-\frac12}(
\nabla \hat f_{k-1}
+\mathcal L\hat\blambda_{k-1})+\sqrt{\delta}\sigma Q^{-\frac12} \bxi_k
\|^2\\
&=
\frac{1}{2}\|Q^{\frac12}\hat\bx_{k-1}-\delta Q^{-\frac12}
\nabla \hat f_{k-1}\|^2
+\frac{\delta^2}{2}\| Q^{-\frac12}\mathcal L\hat\blambda_{k-1}\|^2
\\ 
& -\delta\inp{\mathcal L\hat\blambda_{k-1}}{\hat\bx_{k-1}-\delta Q^{-1}
\nabla \hat f_{k-1}}
 +\frac{\delta\sigma^2}{2}\|Q^{-\frac12}\xi_k\|^2
+I(\xi_k).
\end{split}
\]
Similarly for the second term in \eqref{eq:lyapunov discrete} we have,
\[
\begin{split}
 &D_{\Psi}(\blambda_k,\blambda^\star)
 =\frac{1}{2}\|\hat \blambda_{k}\|^2_{R}\\
&=\frac{1}{2}\| R^{\frac12}\hat \blambda_{k-1}
+\delta R^{-\frac12}\cL\hat \bx_{k}\|^2\\
&=\frac{1}{2}\| R^{\frac12}\hat \blambda_{k-1}\|^2+
\frac{\delta^2}{2}\| R^{-\frac12}\cL\hat \bx_{k}\|
+\delta\inp{\cL\hat\blambda_{k-1}}{\hat\bx_k}\\
&=\frac{1}{2}\| R^{\frac12}\hat \blambda_{k-1}\|^2
+\frac{\delta^2}{2}\| R^{-\frac12}\cL\hat \bx_{k}\|
+\delta \inp{\cL\hat\blambda_{k-1}}
{\hat\bx_{k-1}-\delta Q^{-1}(\nabla \hat f_{k-1}+\cL\hat\blambda_{k-1})}+I(\xi_k).
\end{split}
\]
Therefore,
\begin{equation}\label{eq:bound lyapunov k}
\begin{split}
D_\Phi(\bx_k,\bx^\star)+D_{\Psi}(\blambda_k,\blambda^\star)
&=
\frac12 \|Q^{\frac12}\hat\bx_{k-1}-\delta Q^{-\frac12}
\nabla \hat f_{k-1}\|^2
-\frac{\delta^2}{2}\| Q^{-\frac12}\mathcal L\hat\blambda_{k-1}\|^2
 \\
 &+\frac12\| R^{\frac12}\hat \blambda_{k-1}\|^2
+\frac{\delta^2}{2}\| R^{-\frac12} \cL\hat\bx_{k}\|^2
+\frac{\delta\sigma^2}{2}\|Q^{-\frac12}\xi_k\|^2
+I(\xi_k).
\end{split}
\end{equation}
It follows from \eqref{eq:sigma upper}, 
\[
\overline\sigma(\cL)^2D_\Phi(\bx_k,\bx^\star)=\frac12\overline\sigma(\cL)^2\|Q^{\frac12}\hat\bx_k\|^2\geq \frac12\| R^{-\frac12}\cL\hat \bx_{k}\|^2.
\]
It follows from Lemma \ref{lemma: range space lambda} $\blambda_{k-1}\in\range(\cL)$ and therefore using \eqref{eq:sigma lower} we obtain,
\[
\underline\sigma(\cL)^2D_{\Psi}(\blambda_{k-1},\blambda^\star)=\frac12  \underline\sigma(\cL)^2\| R^{\frac12}\hat\blambda_{k-1}\|^2\leq\frac12  \|Q^{-\frac12}\cL\hat \blambda_{k-1}\|^2.
\]
We use the last two inequalities to bound \eqref{eq:bound lyapunov k} as follows,
\begin{align}\label{eq:bound lyapunov k2}
\left(1-\delta^2\overline\sigma(\cL)^2\right)D_\Phi(\bx_k,\bx^\star)+D_{\Psi}(\blambda_k,\blambda^\star)
&\leq 
\frac12 \|Q^{\frac12}\hat\bx_{k-1}-\delta Q^{-\frac12}
\nabla \hat f_{k-1}\|^2 \\
&+(1-\delta^2\underline\sigma(\cL)^2)D_{\Psi}(\blambda_{k-1},\blambda^\star)\\
&+\frac{\delta\sigma^2}{2}\|Q^{-\frac12}\xi_k\|^2
+I(\xi_k).
\end{align}
We can bound the first term in the inequality above using \eqref{eq: smooth quad case} from Lemma \ref{lemma:quad map inqualities},
\[
\begin{split}
 \frac12\|Q^{\frac12}\hat\bx_{k-1}-\delta Q^{-\frac12}
\nabla \hat f_{k-1}\|^2 
&\leq D_\Phi(\bx_{k-1},\bx^\star)+\left(\frac12\delta^2 L_f'-\delta\right)\inp{\nabla \hat f_{k-1}}{\hat\bx_{k-1}}\\
& \leq (1+\delta^2\mu_fL'_f-2\delta\mu_f)D_\Phi(\bx_{k-1},\bx^\star)\\
&=\rho_x(1-\delta^2\overline\sigma(\cL)^2)D_\Phi(\bx_{k-1},\bx^\star)
+(\rho_x\delta^2\overline\sigma(\cL)^2-\delta\mu_f)D_\Phi(\bx_{k-1},\bx^\star)\\
&\leq \rho_x(1-\delta^2\overline\sigma(\cL)^2)D_\Phi(\bx_{k-1},\bx^\star),
\end{split}
\]
where 
in the second inequality we used \eqref{eq: sc quad case} and the step-size condition \eqref{eq:stepsize condition} that implies that $(\frac12\delta^2 L_f'-\delta)<0$.  
In the last inequality we used the fact that the step-size condition \eqref{eq:stepsize condition} means that $0<\rho_x<1$ and that the last term is negative.  
Using the bound above in \eqref{eq:bound lyapunov k2} and the definition of in \eqref{eq:lyapunov discrete},
\[
\begin{split}
V(\bx_k,\blambda_k)
&\leq \rho_x(1-\delta^2\overline\sigma(\cL)^2)D_\Phi(\bx_{k-1},\bx^\star)
+\rho_\lambda D_{\Psi}(\blambda_{k-1},\blambda^\star)+\frac{\delta\sigma^2}{2}\|Q^{-\frac12}\xi_k\|^2
+I(\xi_k) \\
&\leq \rho V(\bx_{k-1},\blambda_{k-1})+\frac{\delta\sigma^2}{2}\|Q^{-\frac12}\xi_k\|^2
+I(\xi_k)
\end{split}
\]
Taking expectation on both sides,
\[
\mathbb E [V(\bx_k,\blambda_k)]
\leq \rho \mathbb E [ V(\bx_{k-1},\blambda_{k-1})]+\frac{\delta\sigma^2\textnormal{tr}(Q^{-1})}{2}.
\]
Iterating the inequality above,
\[
\mathbb E [V(\bx_k,\blambda_k)]
\leq \rho^{k} \mathbb E [V(\bx_{0},\blambda_{0})]+\frac{\delta\sigma^2\textnormal{tr}(Q^{-1})}{2}
\sum_{i=0}^{k-1}\rho^k
\leq \rho^{k} \mathbb E [V(\bx_{0},\blambda_{0})]
+\frac{\delta\sigma^2\textnormal{tr}(Q^{-1})}{2(1-\rho)}
\]
where in the last inequality we used the fact that $\sum_{i=0}^{k-1}\rho^k\leq \sum_{i=0}^{\infty}\rho^k=1/(1-\rho)$. Finally note that the step-size condition \eqref{eq:stepsize condition} means that $0<\rho<1$
\end{proof}
The convergence rate in \eqref{eq:convergance discrete} makes explicit the effect of the condition number of the objective function ($L_f/\mu_f$) and the condition number of the graph 
($\overline\sigma(\cL)/\underline\sigma(\cL)$).
We illustrate this point in  \ref{fig:discrete constnants}. 
In \ref{fig:discrete constnants} we plot the convergence rate $\rho$ as a function of the objective function's condition number. 
We set $Q=R=I$ and plot two different cases: (a) a fully connected graph ($\overline\sigma(\cL)/\underline\sigma(\cL)=1$) and (b) a graph with a less favorable condition number ($\overline\sigma(\cL)/\underline\sigma(\cL)=3$).
When the graph is well conditioned (solid blue line) we see that $\rho$ depends only on the condition number of the objective function. However, when the graph is not fully connected (dotted red line) we see that the convergence rate is independent of the objective function properties and is dominated by the condition number of the graph. After a certain threshold (indicated by the dotted vertical line, $L_f/\mu_f\approx 4.5$), the convergence rate is again dominated by the properties of the objective function.
This example illustrates the effect of the graph's spectral properties on the convergence rate.
As discussed previously setting $Q=I$, and $R=\cL_\beta Q^{-1}\cL_\beta$ eliminates the dependence of the convergence rate to the graph's spectral properties.
Other choices for $\Phi$, e.g. $\Phi=\nabla^2f$ means that $\Psi$ has a dual job of preconditioning both the primal and Lagrangian dual variables. We note that a graph condition number of $\overline\sigma(\cL)/\underline\sigma(\cL)=3$ is not particularly high. For example an undirected  Erd{\"o}s-R{\`e}nyi graph with $N=100$ nodes and a probability of connection $0.25$ has an average condition number of approximately $3$.

\begin{figure}
\centering\includegraphics[width=0.6\textwidth]{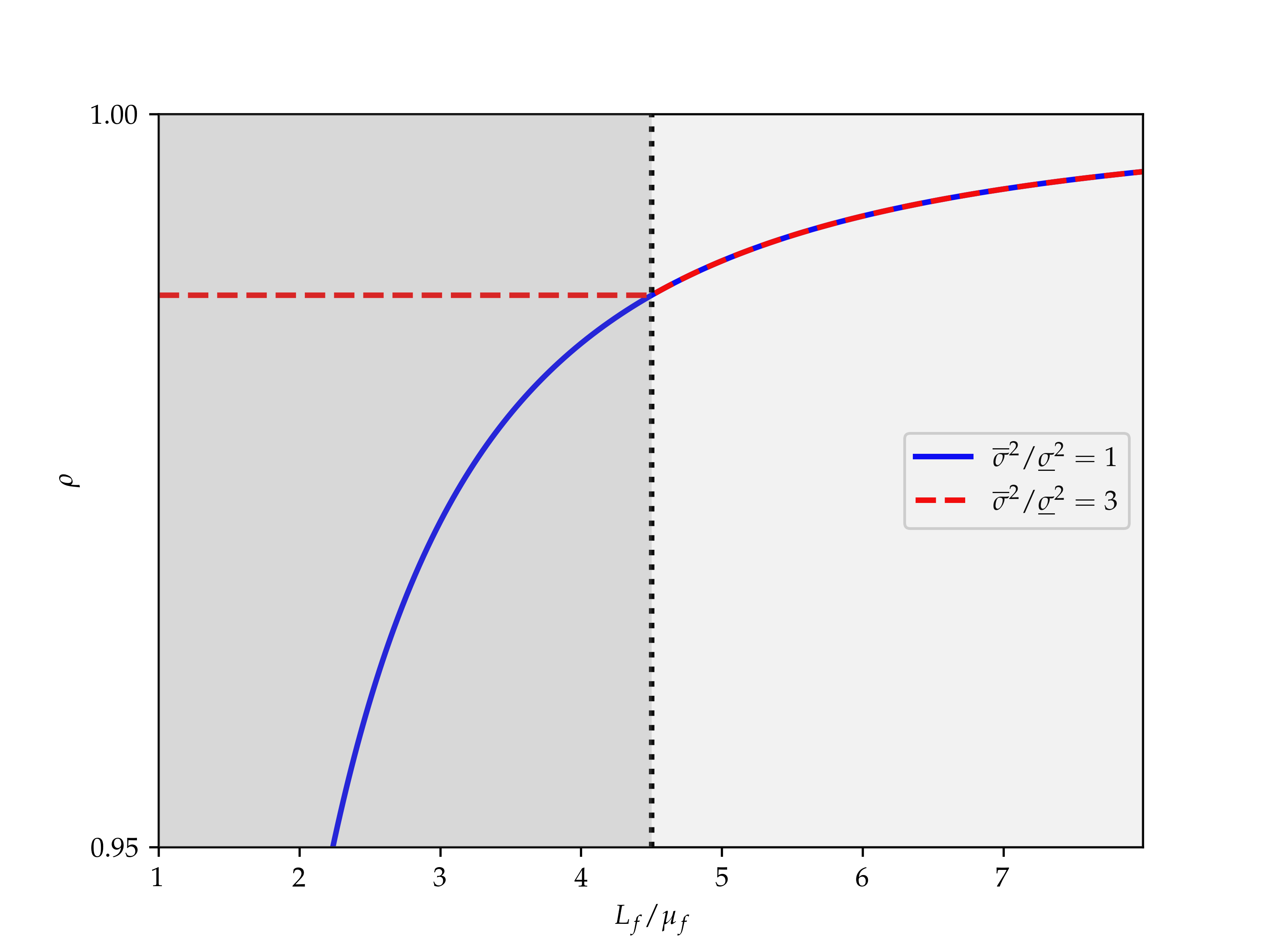}
 \caption{The algorithms convergence rate without preconditioning $Q=R=I$. 
 When the graph is well-conditioned ($\overline\sigma^2/\underline\sigma^2=1$) the convergence rate of the algorithm (solid blue line) is dominated by the problem's condition number $L_f/\mu_f$.
 When the graph is ill-conditioned ($\overline\sigma^2/\underline\sigma^2=3$) the convergence rate (dotted red line) is dominated by the spectral properties of the graph up to a certain condition number (in this case $L_f/\mu_f\approx 4.5$), after this point the convergence rate is again dominated by the model's condition number.  
 \label{fig:discrete constnants}}
\end{figure}

\begin{figure}
\captionsetup[subfigure]{justification=Centering}
\begin{subfigure}[t]{0.3\textwidth}
    \includegraphics[width=\textwidth]{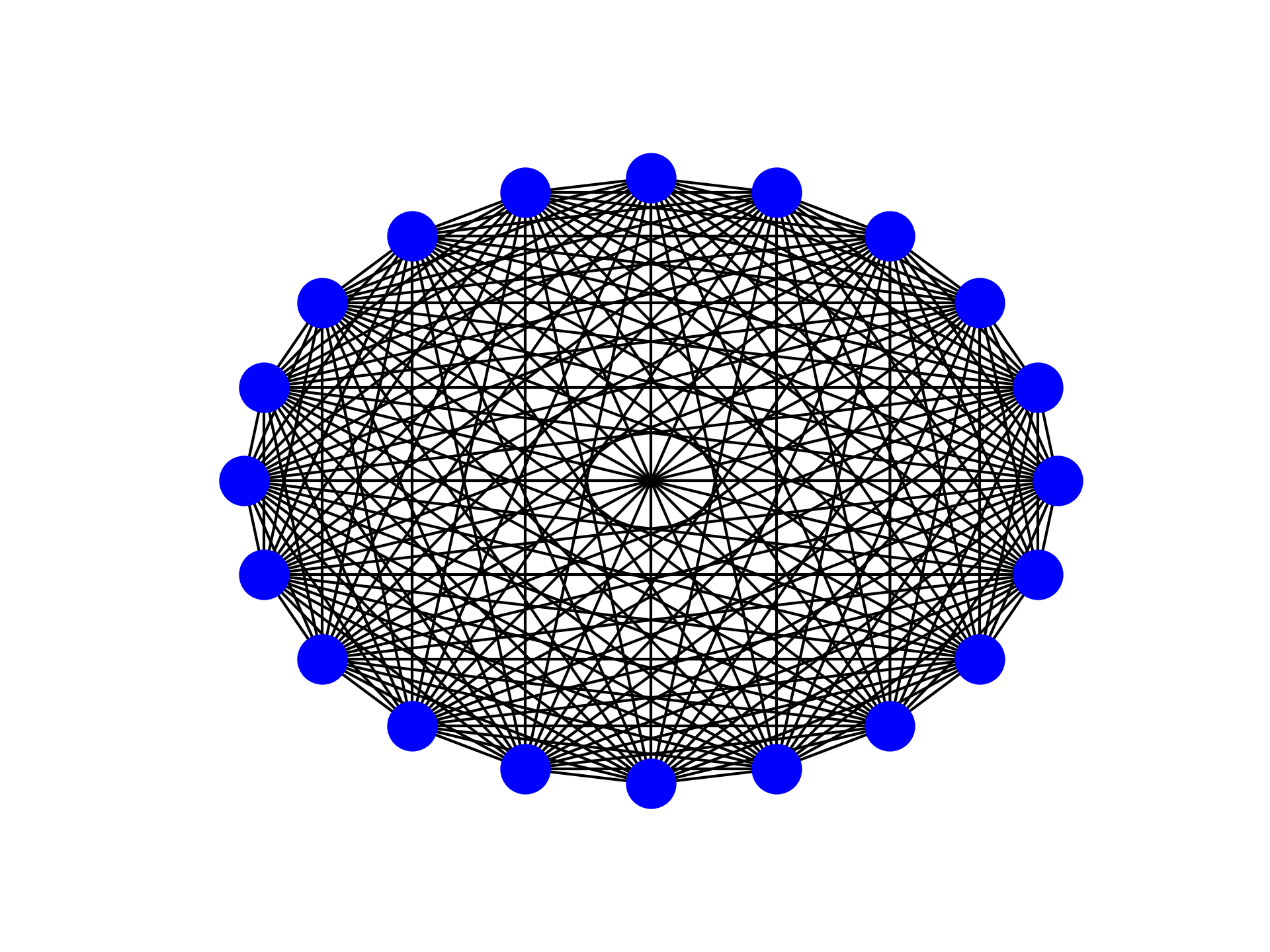}
    \caption{\label{fig:fc graph}A fully connected graph with $N=20$ nodes}
\end{subfigure}\hspace{\fill} 
\begin{subfigure}[t]{0.3\textwidth}
    \includegraphics[width=\linewidth]{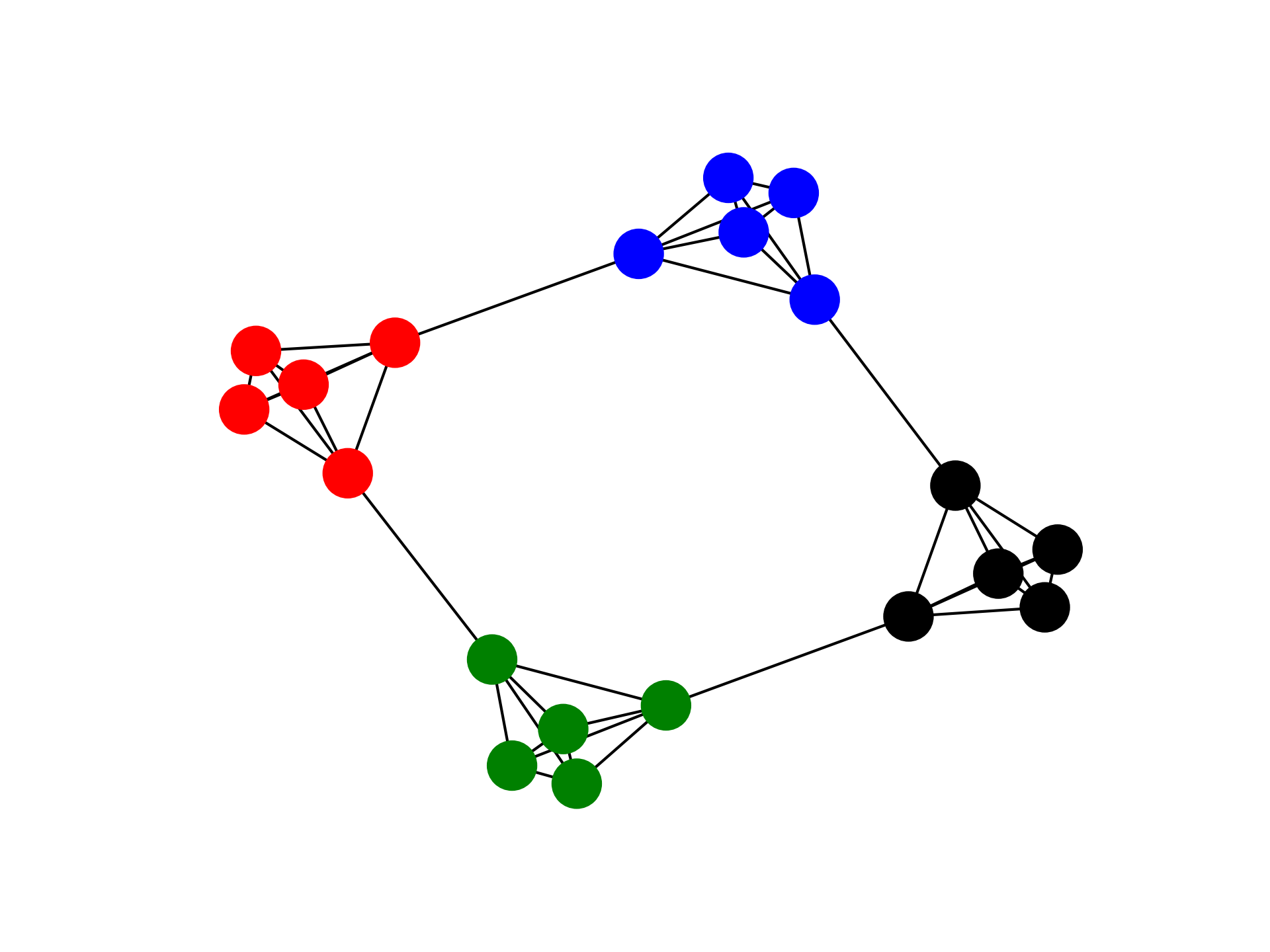}
    \caption{\label{fig:ring graph}
    A ring-of-cliques graph with $N=20$ nodes and $4$ cliques}
\end{subfigure}
\begin{subfigure}[t]{0.3\textwidth}
    \includegraphics[width=\linewidth]{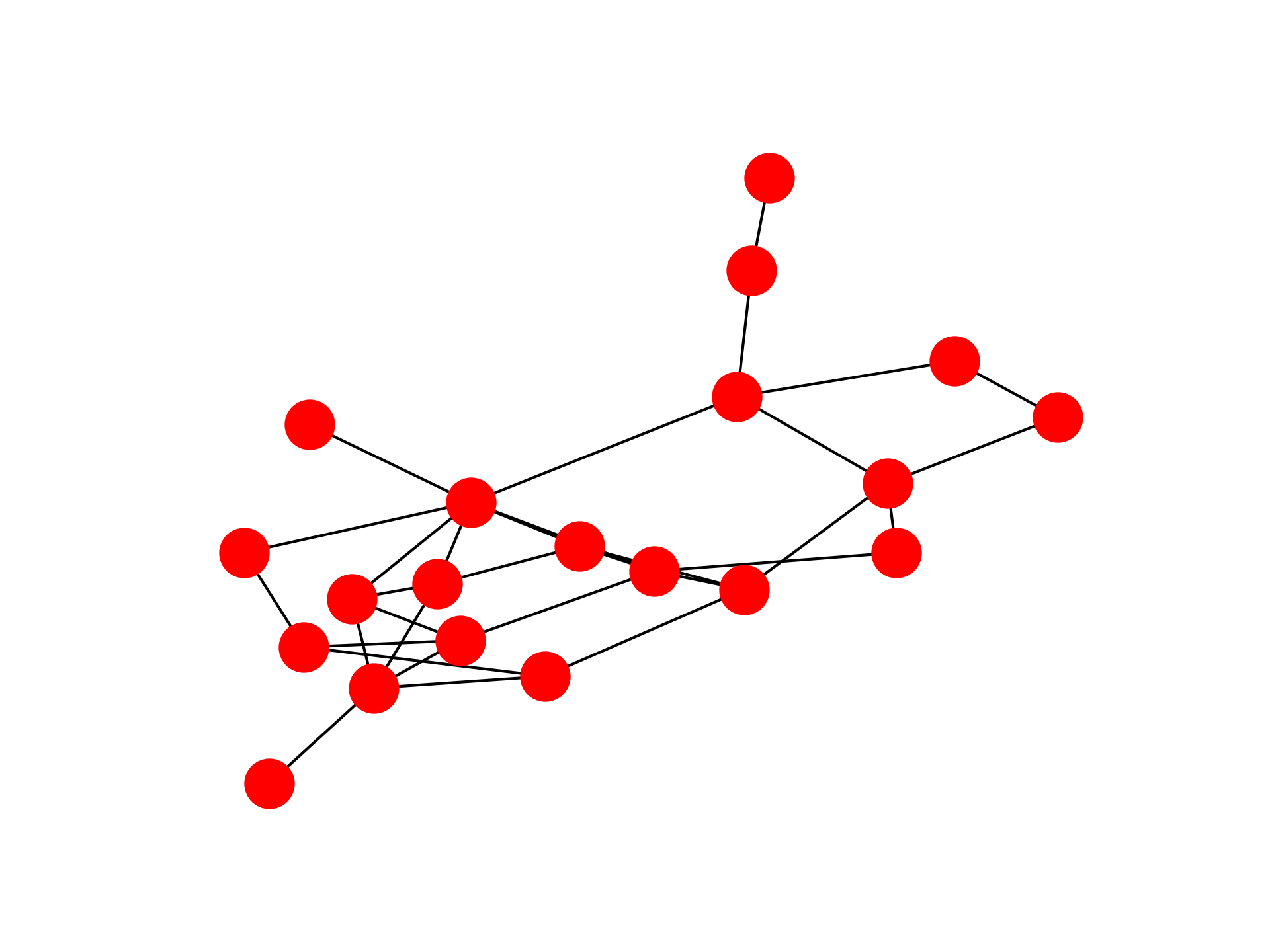}
    \caption{\label{fig:erdos graph}
    An Erd{\H o}s-R{\`e}nyi graph with $N=20$ nodes and probability of connection $p=0.2$}
\end{subfigure}
\caption{
Different types of graphs considered in the numerical experiments.
\label{fig: graphs numerical}
}
\end{figure} 

\section{Numerical Results}\label{sec:num}\color{black}
In this section we illustrate the performance of the proposed algorithm in different settings.
Our intention is to illustrate that the proposed method can be used to reduce and sometimes eliminate the impact of the graph's spectral properties.
All numerical experiments are run on a modern laptop with 64Gb memory and an 8-core CPU. 
Because we implemented several algorithms with different iteration cost we report CPU-times as opposed to the number of iterations required to find a solution.

The test problem is given by the following quadratic optimization problem,
\begin{equation}\label{eq:test problem}
\min_{x_i\in\mathcal X}\sum_{i=1}^N \|A_ix_i-b_i\|^2_2 \quad \text{s.t.} \cL\bx=0.
\end{equation}
Where $\mathcal X\subset\R^d$, $A\in\R^{m\times d}$, and $b_i\in\R^m$. 
In the simulations below we take $d=50$, $M=75$ and $N=60$.
The matrices $A_i$ are generated randomly with a condition number of $\rho=2$. 
We also generate an ill-conditioned variant of the problem in \eqref{eq:test problem} by taking the absolute value of the entries in $A_i$ and $b_i$. 
For the graph associated with the problem above we consider three different cases: (a) a fully connected graph, (b) a ring-of-cliques graph (with 5 cliques of 12 nodes) and (c) an Erd{\H o}s-R{\`e}nyi graph. 
For $N=20$ the three cases are illustrated in \ref{fig:quad_unc_cl_ill_det}. 
For the experiments below we use EPISMD$-(\nabla^2 \Phi,\nabla^2 \Psi)$ to denote how \eqref{eq: gauss-seidel zm} is implemented.
For example EPISMD$-(\nabla^2 f,I)$ denotes that we use $\nabla^2 \Phi=\nabla^2 f$ and that we do not precondition the Lagrangian dual variables.
Whenever we precondition the Lagrangian dual variables we use $R=\cL_\beta(\nabla^2f)^{-1}\cL_\beta$ with $\beta=0.0001$. 
For reasons of computational efficiency we use $R^{-1}=\cL^{-1}_\beta(\nabla^2f)\cL^{-1}_\beta$ in our implementation.
This is because we can compute $\cL^{-1}_\beta$ using a one-off cost of an eigenvalue decomposition of an $N$ dimensional matrix. The Hessian of $f$ can be computed at a cost of $O(Nd^2)$.
Assessing the efficiency or relative benefit of computing the additional preconditioner is problem dependent, but we note this is only computed once at initialization. As will be illustrated in the numerical examples below, when the graph is ill-conditioned then it can improve the performance by orders of magnitude. 
We also note that there are a lot of ways to improve the way we compute the preconditioner. 
For example, we could use quasi-Newton style methods, statistical preconditioning (see \cite{hendrikx2020statistically}), multilevel methods (e.g. \cite{tsipinakis2021multilevel}) and specialized techniques for solving equations involving the Laplacian (see \cite{MR3071502}).

Below we compare several algorithms with constant step-sizes. In order to be fair we tuned the step-size of all algorithms separately for best performance. Only the Distributed Projected Gradient Descent (DPGD) algorithm uses a diminishing step-size strategy, and we used the schedule described in \cite{MR3680426}. The exact solution was computed using CVXPY \cite{diamond2016cvxpy}.

\subsection{The Unconstrained Case}\label{sec:num unconstrained}
In the first set of numerical experiments we consider the unconstrained case, i.e. we take $\mathcal X=\R^d$ in \eqref{eq:test problem}.
Since the problem is unconstrained we compare against the state-of-the-art for unconstrained problems. 
In our case we compare against the Distributed Stochastic Gradient Tracking (DSGT) method from \cite{pu2021distributed}.
DSGT is exact, has a linear convergence rate and can handle noisy gradient evaluations. In \cite{pu2021distributed} it was compared against other algorithms and was found to outperform many of them. 
We also implemented the Distributed Augmented Lagrangian Method (DALM) analyzed in \cite{mateos2014p}.
DALM is a good baseline to compare against since in the unconstrained case our algorithm reduces to the DALM if we take $Q=R=I$. 
Therefore any improvement against DALM can be attributed to our choice of mirror maps. 

In our first experiment we consider a fully connected network, $\sigma=0$, and a well conditioned problem. 
As can be seen from \ref{fig:quad_unc_fc_rho2_det} all algorithms converge to the same solution, any small variations are due to rounding errors.
In our second experiment we change the condition number of the problem as described above. 
In \ref{fig:quad_unc_fc_ill_det} we observe that EPISMD$-(\nabla^2 f,I)$ and EPISMD$-(\nabla^2 f,\nabla^2\Psi)$ outperform the other methods.
This observation is not surprising since the problem is ill-conditioned and EPISMD$-(\nabla^2 f,I)$ and EPISMD$-(\nabla^2 f,\nabla^2\Psi)$ precondition the primal variables.
We also see a small advantage of EPISMD$-(\nabla^2 f,\nabla^2\Psi)$. This is because $\nabla^2\Psi$ also contains information from the Hessian of $f$.  
In our next set of experiments we analyze the impact of using a ring-of-cliques graph (we used twelve cliques with five nodes in each clique).
In \ref{fig:quad_unc_cl_rho2_det} we plot the results for the case where the problem is well-conditioned and in \ref{fig:quad_unc_cl_ill_det} we plot the case where the problem is also ill-conditioned. 
It is clear from these two figures that the impact of the graph's spectral properties is significant for all methods except for EPISMD$-(\nabla^2 f,\nabla^2\Psi)$.
This is of course to be expected given the results of the previous section. 
When the problem is unconstrained EPISMD$-(\nabla^2 f,\nabla^2\Psi)$ pre-conditions both the primal and dual variables and can therefore mitigate ill-conditioning that originates from either the model, the graph or both. Note that in the experimental setup of this section EPISMD$-(\nabla^2 f,\nabla^2\Psi)$ is a preconditioned version of DALM. 
Therefore it would be interesting to develop a variant of DSGT that also adapts to the geometry of the graph's Laplacian. 
\begin{figure}
\captionsetup[subfigure]{justification=Centering}
\begin{subfigure}[t]{0.45\textwidth}
    \includegraphics[width=\textwidth]{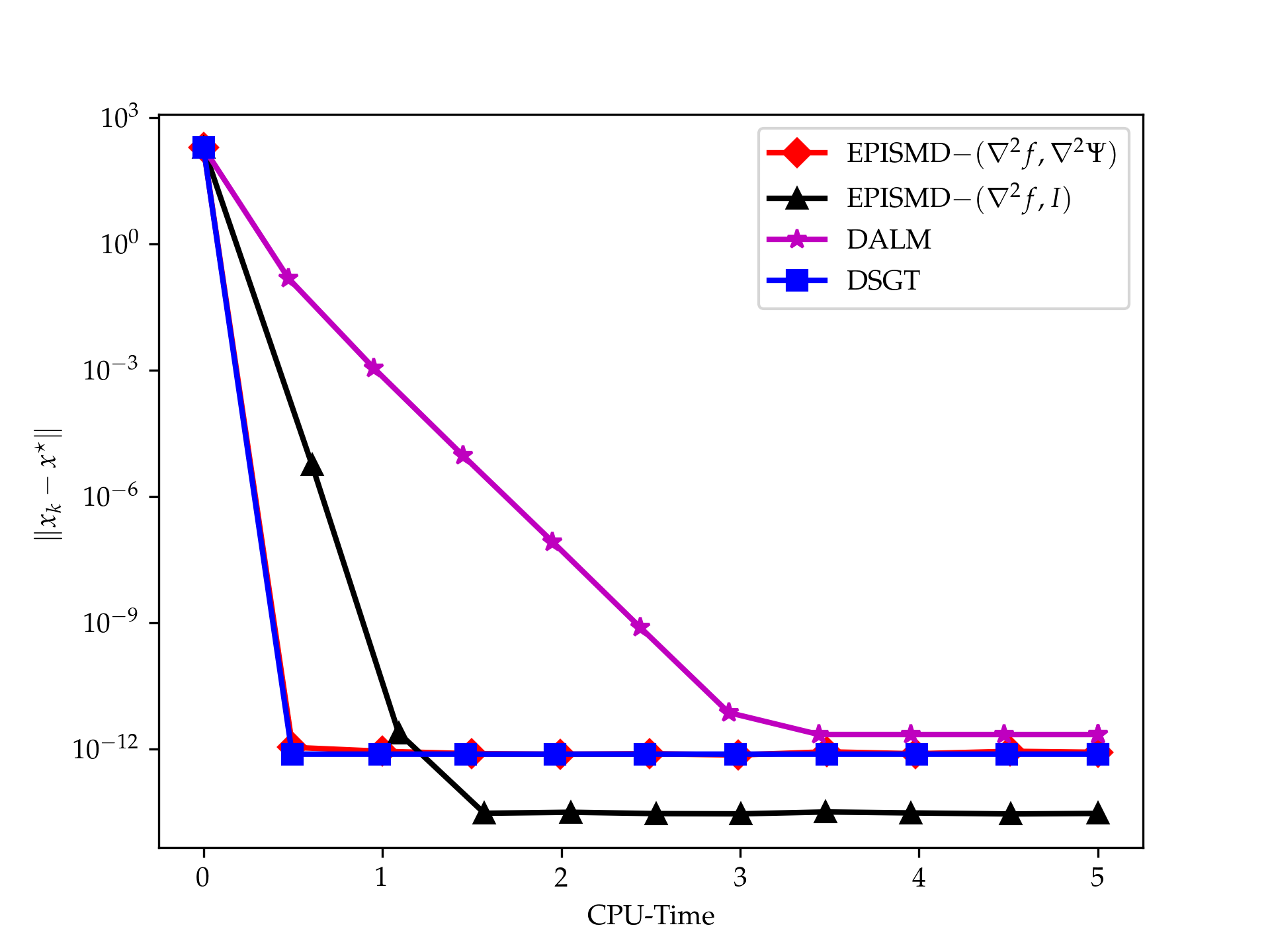}
    \caption{\label{fig:quad_unc_fc_rho2_det}Well conditioned model with a fully connected graph.}
\end{subfigure}\hspace{\fill} 
\begin{subfigure}[t]{0.45\textwidth}
    \includegraphics[width=\linewidth]{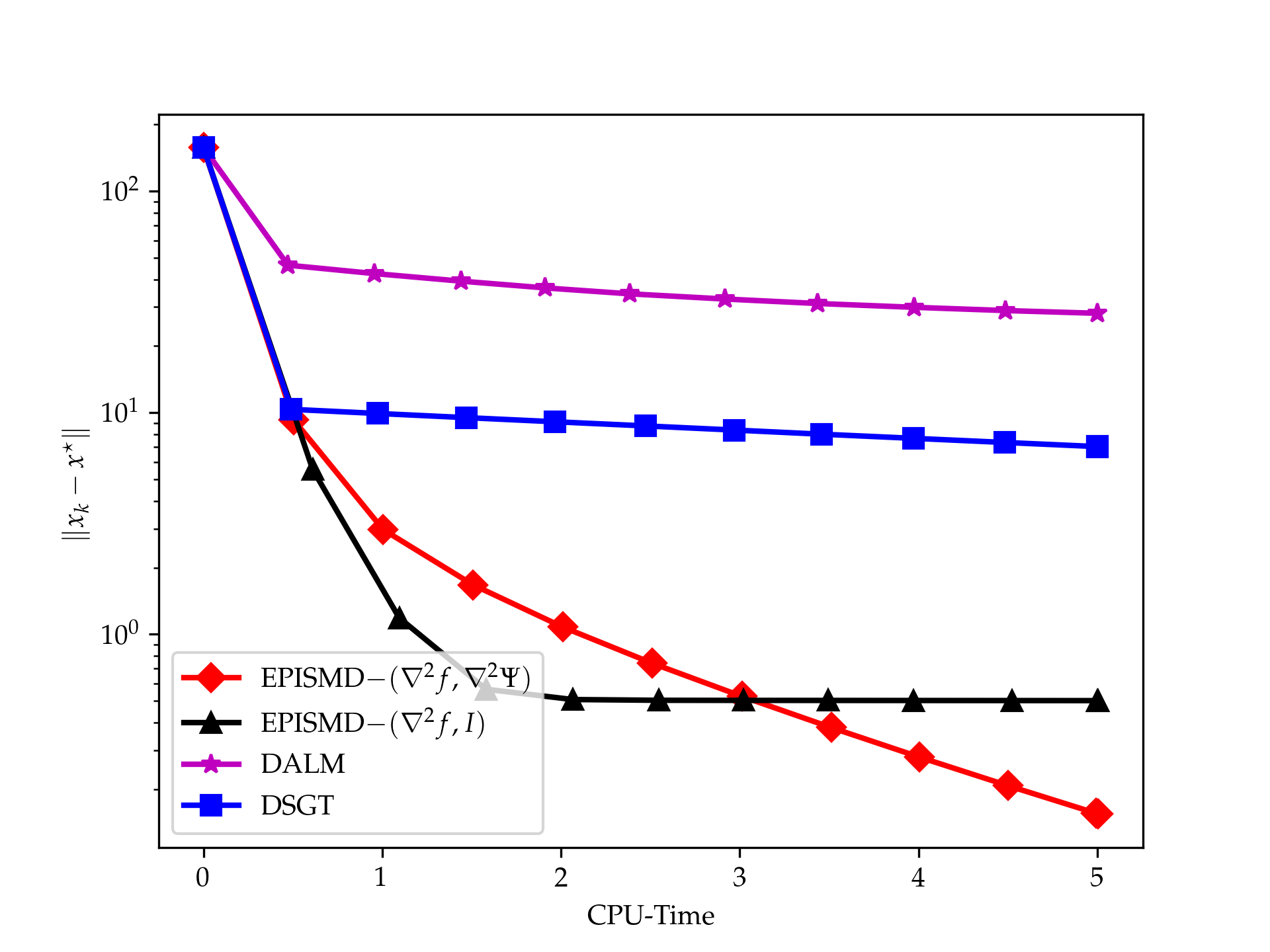}
    \caption{\label{fig:quad_unc_fc_ill_det}
    Ill conditioned model with a fully connected graph.}
\end{subfigure}
\bigskip 
\begin{subfigure}[t]{0.45\textwidth}
    \includegraphics[width=\linewidth]{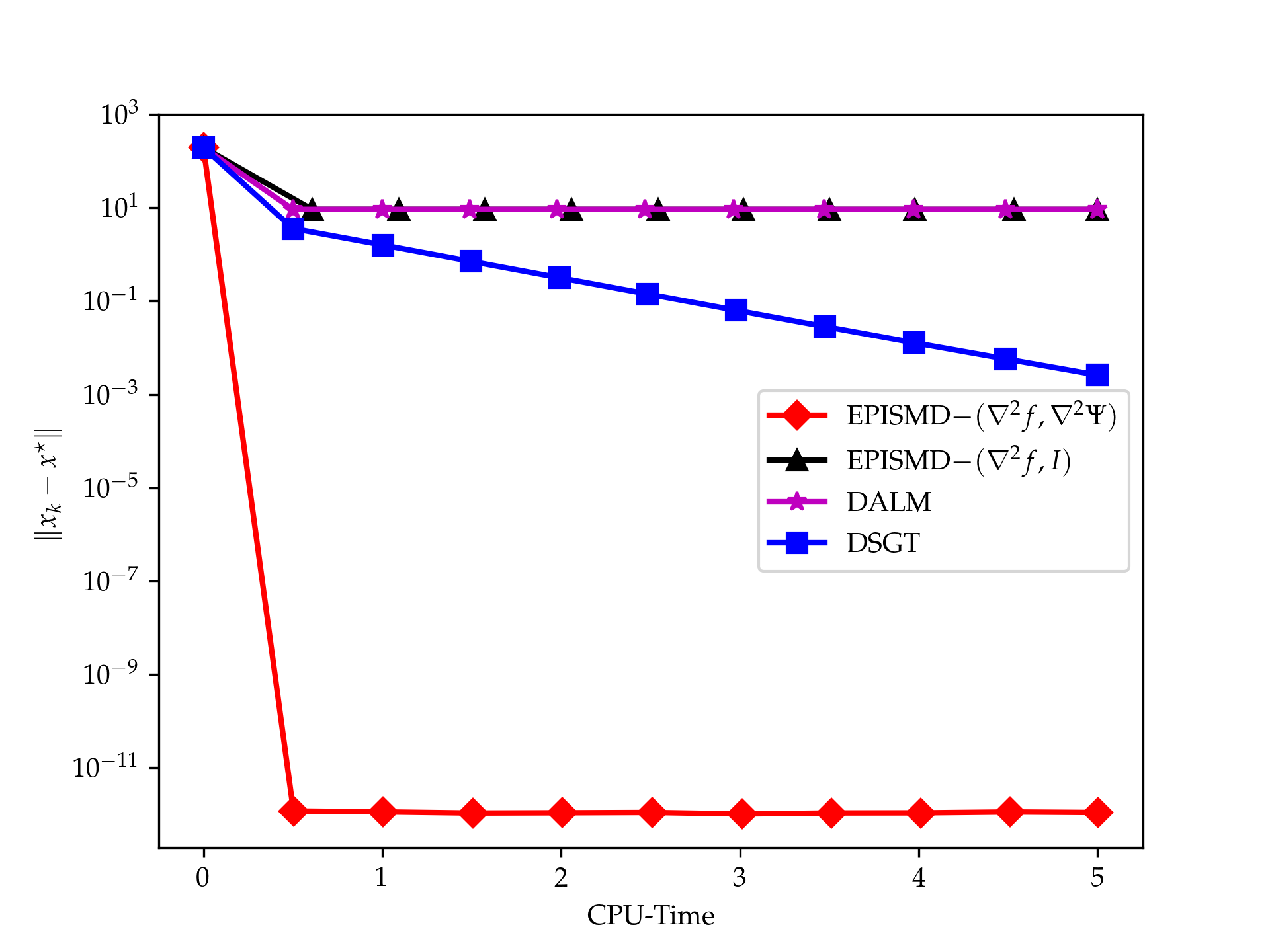}
    \caption{Well conditioned model on a ring-of-cliques graph.\label{fig:quad_unc_cl_rho2_det}}
\end{subfigure}\hspace{\fill} 
\begin{subfigure}[t]{0.45\textwidth}
    \includegraphics[width=\linewidth]{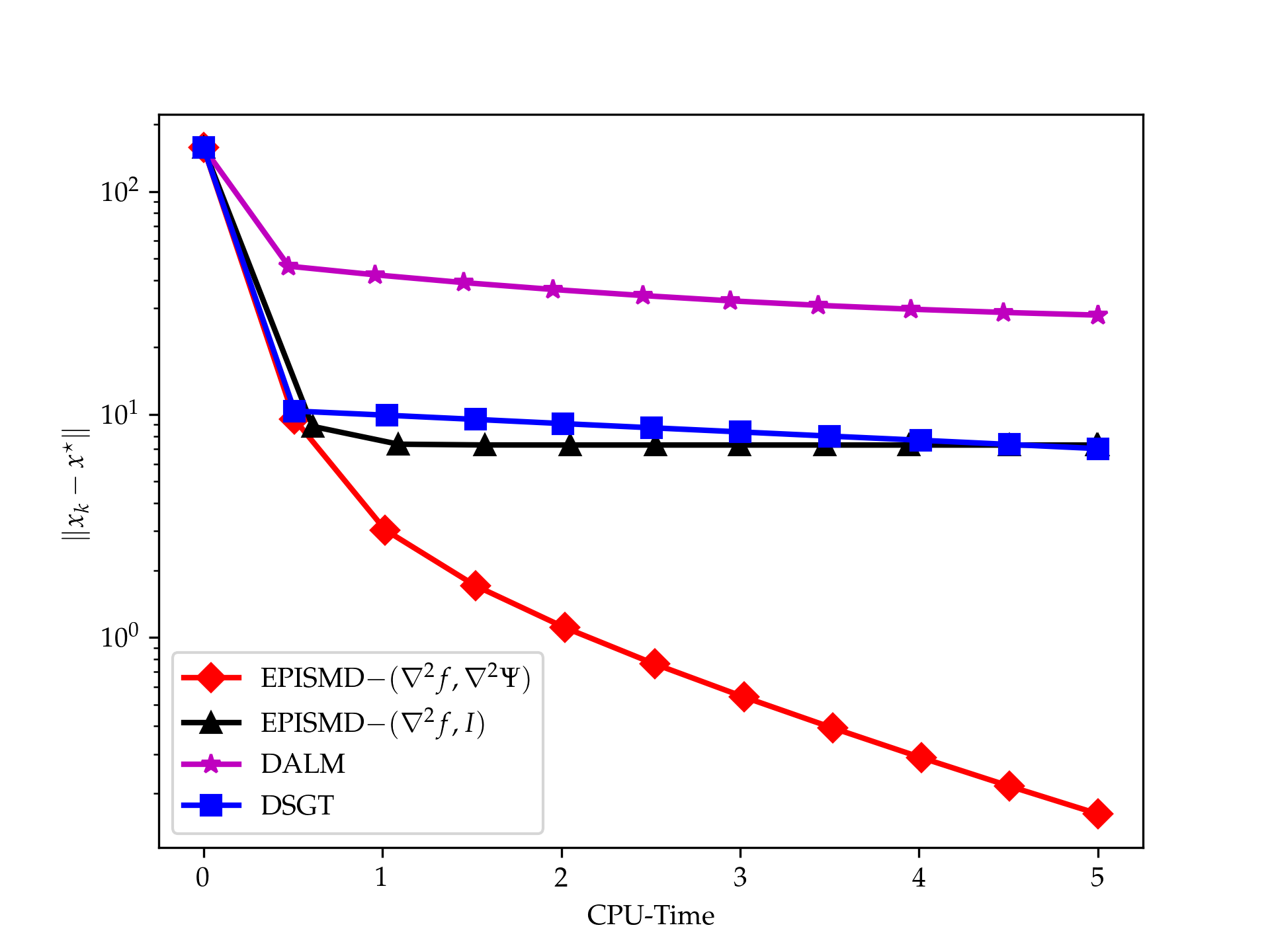}
    \caption{Ill conditioned model on a ring-of-cliques graph.\label{fig:quad_unc_cl_ill_det}}
\end{subfigure}
\caption{Numerical comparisons of distributed optimization algorithms for unconstrained problems.
(a) When both the problem and graph are well-conditioned the proposed method performs similarly to other methods.
(b-d) However when the model and/or graph are ill-conditioned then the proposed algorithm provides a significant improvement over other methods. 
\label{fig:unconstrained case}}.
\end{figure}

\subsection{The Constrained Case}
In this section we repeat the experiments from \ref{sec:num unconstrained} but when $\mathcal X=\Delta$, where $\Delta$ denotes the $d$-dimensional simplex.
In this case we take $\Phi$ to be the negative entropy $\Phi(x)=\sum_{i=1}^d \ln(x_i)$. 
Again EPISMD-$(\nabla^2\Phi,\nabla^2\Psi)$ is implemented with $\nabla^2\Psi=\cL_\beta\nabla^2f\cL_\beta$. 
As a benchmark we use the Distributed Projected Gradient Method (DPGD) method described in \cite{MR3680426}. 
Note that the DSGT and DPALM from the previous section cannot handle constraints.  We repeat the same set of four experiments from \ref{sec:num unconstrained} and plot the results in \ref{fig: constrained results}.
The results are similar as in the unconstrained case, but the impact of using the proposed preconditioner in conjunction with the conventional mirror maps for the primal variables is even more significant. The results demonstrate that using the proposed mirror map on the Lagrangian dual variables has a significant impact when the graph is ill conditioned.

\begin{figure}
\captionsetup[subfigure]{justification=Centering}
\begin{subfigure}[t]{0.45\textwidth}
    \includegraphics[width=\textwidth]{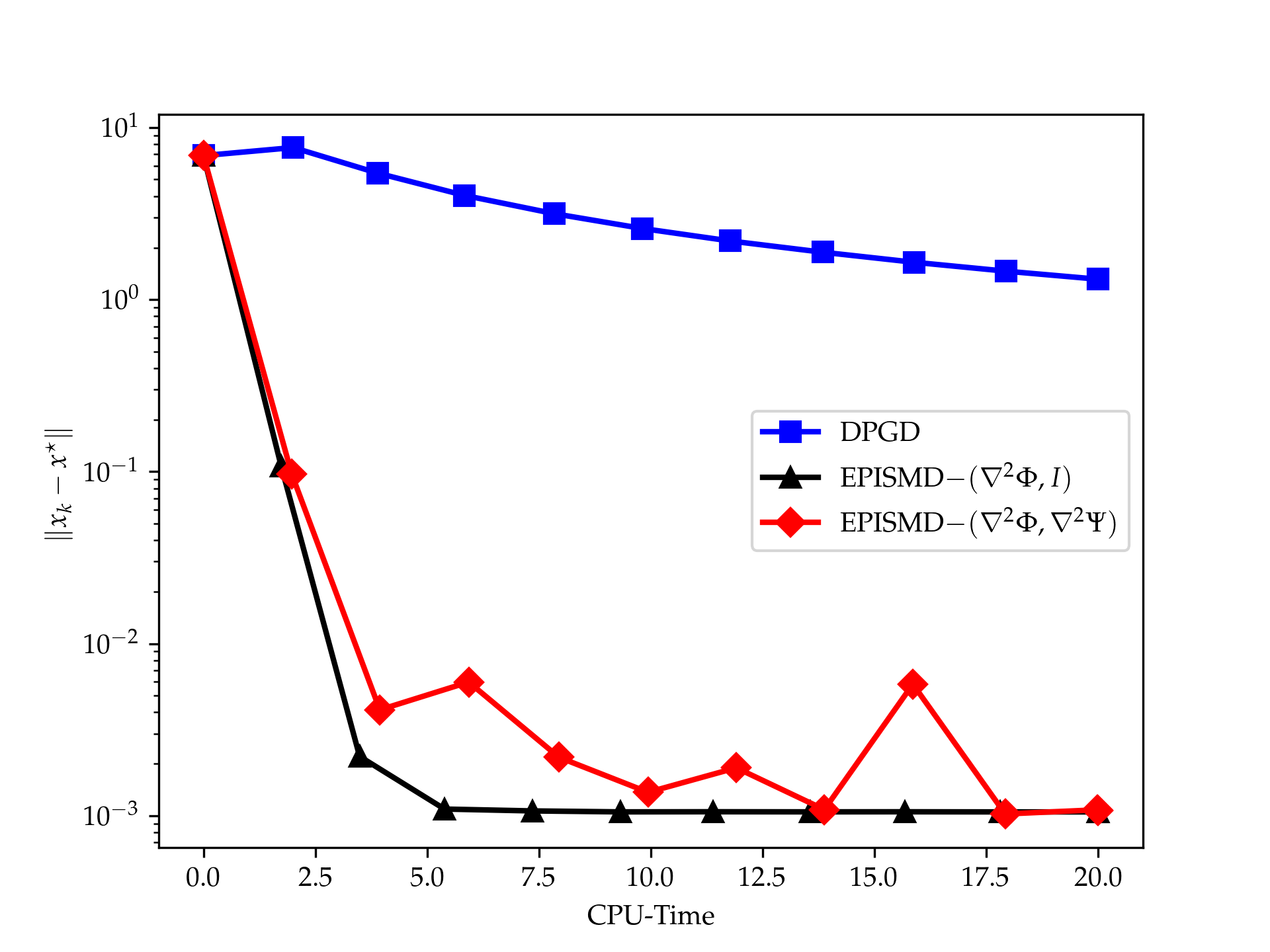}
    \caption{\label{fig:quad_con_con_rho2_det}Well conditioned model with a fully connected graph.}
\end{subfigure}\hspace{\fill} 
\begin{subfigure}[t]{0.45\textwidth}
    \includegraphics[width=\linewidth]{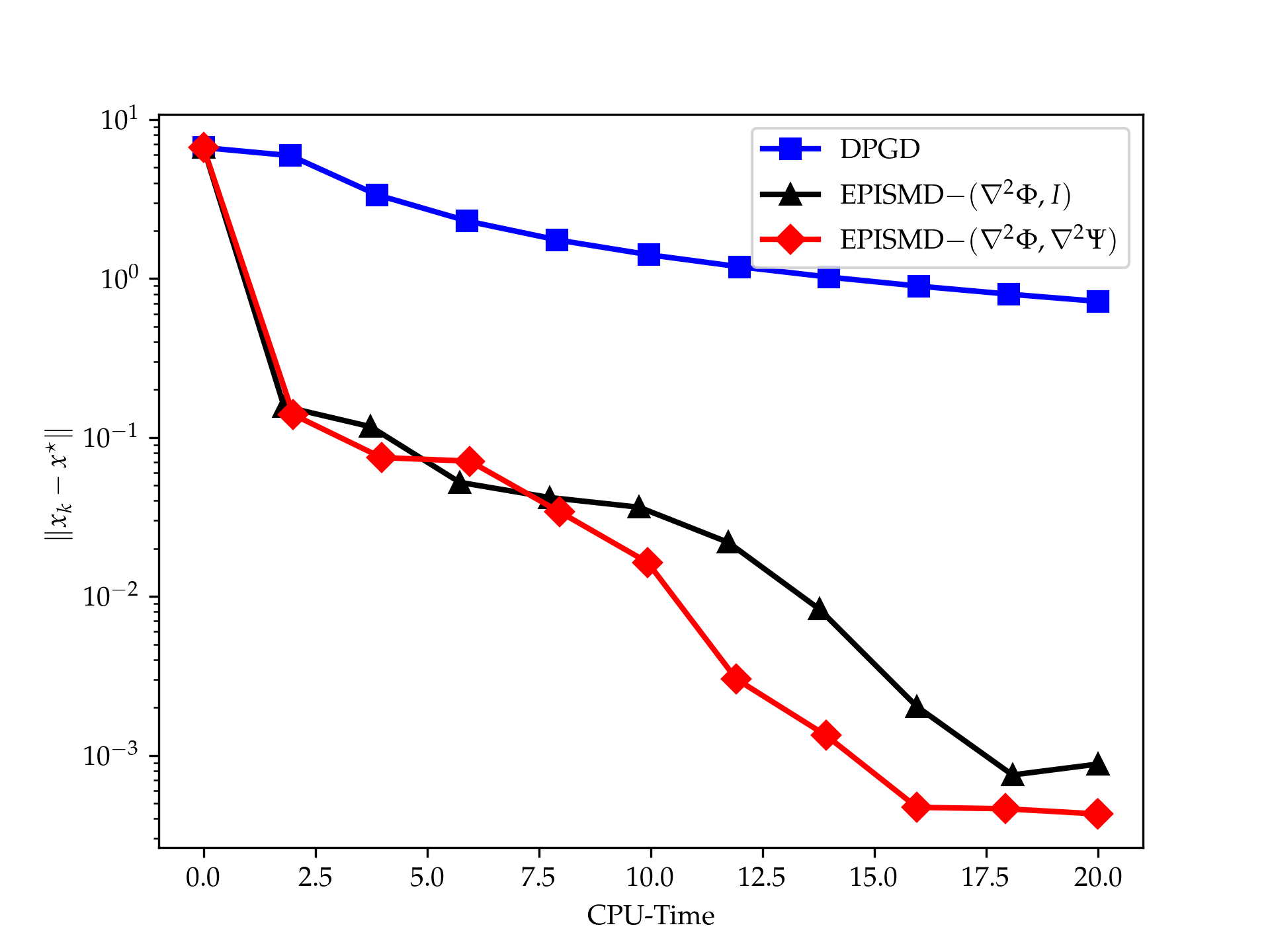}
    \caption{\label{fig:quad_con_fc_ill_det}
    Ill conditioned model with a fully connected graph.}
\end{subfigure}
\bigskip 
\begin{subfigure}[t]{0.45\textwidth}
    \includegraphics[width=\linewidth]{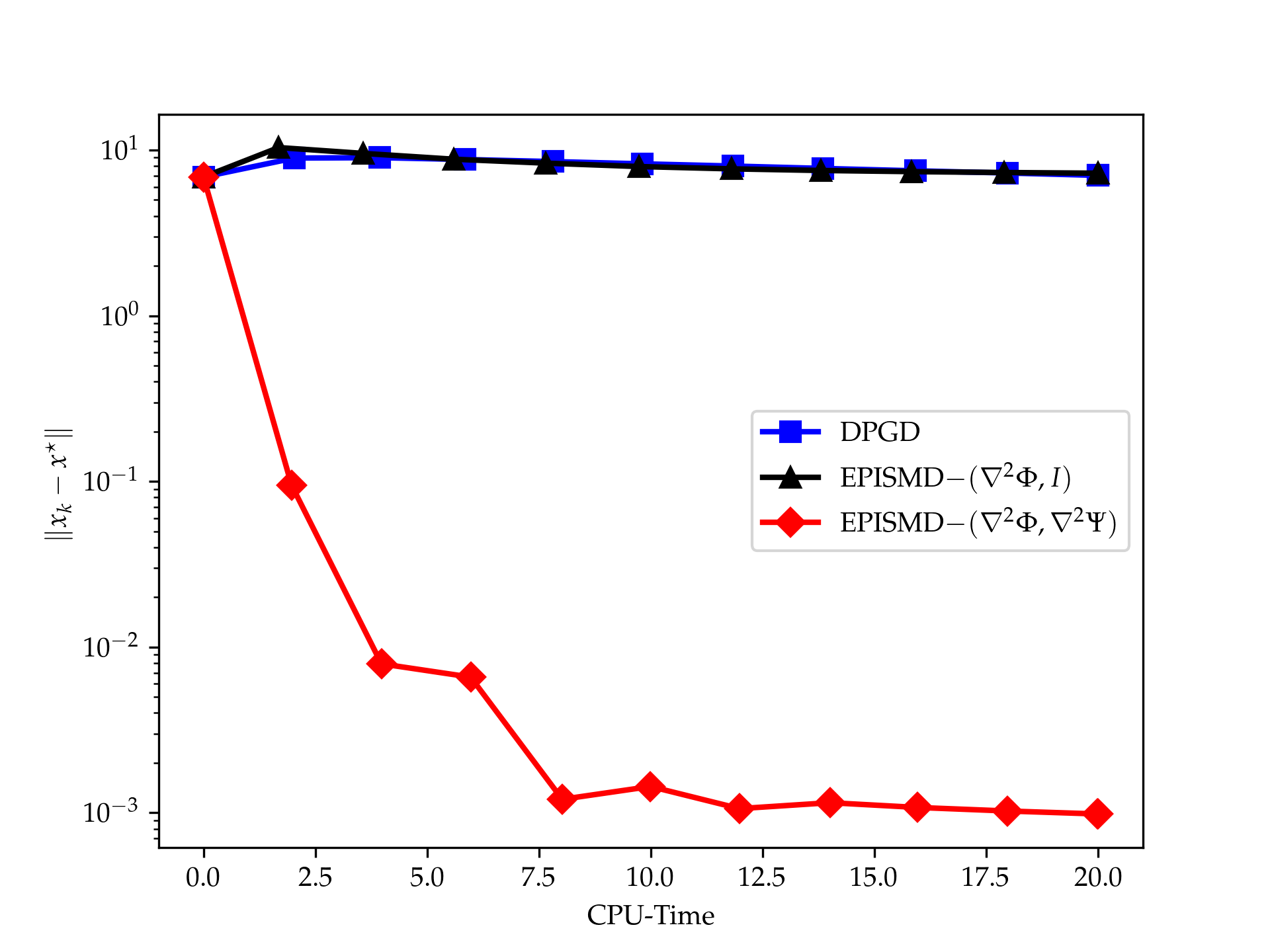}
    \caption{Well conditioned model with a ``ring of cliques" graph.\label{fig:quad_con_cl_rho2_det}}
\end{subfigure}\hspace{\fill} 
\begin{subfigure}[t]{0.45\textwidth}
    \includegraphics[width=\linewidth]{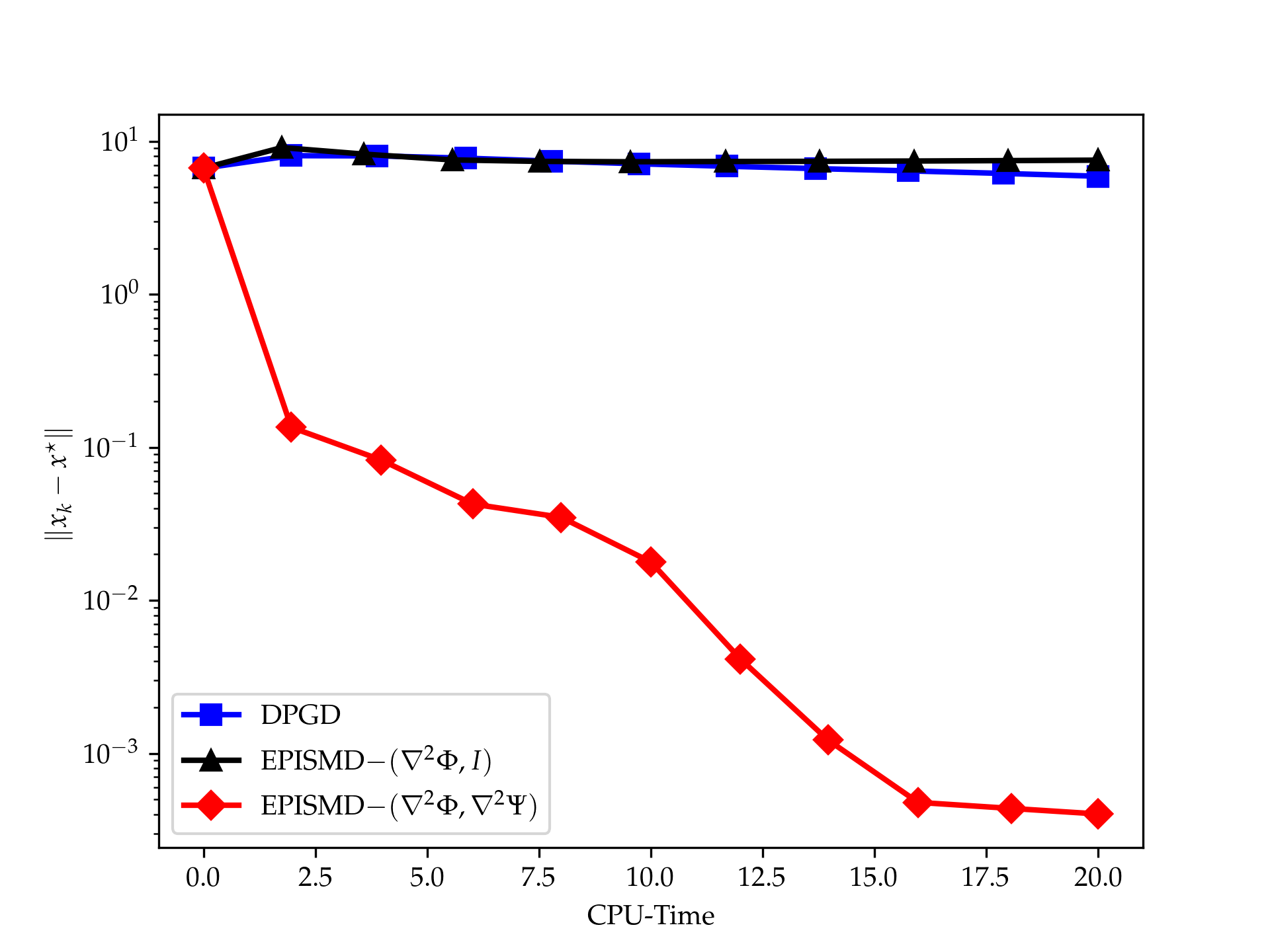}
    \caption{Ill conditioned model with a ``ring of cliques" graph.\label{fig:quad_con_cl_ill_det}}
\end{subfigure}
\caption{Numerical comparisons of distributed optimization algorithms for simplex-constrained problems.
Like the unconstrained case in \ref{fig:unconstrained case} the proposed method is particularly efficient when the graph is ill-conditioned.  
\label{fig: constrained results}
}
\end{figure}

\subsection{The Effect of Noise} As shown in \ref{prop:expconv_stoch_dual} (continuous time) and \ref{prop: discretization main} (discrete time) the effect of noise is an extra error term that depends on $\sigma$. 
To show that the conclusions of the numerical experiments above are still valid for the stochastic case we re-run all the experiments using different levels of noise.
To save space we only show one particular instance of the model in \eqref{eq:test problem}. 
The instance we consider is when the graph is a ring of cliques, the constraint is the $d-$dimensional simplex and the problem is ill-conditioned. In other words, we used the same experimental setup used to produce \ref{fig:quad_con_cl_ill_det} but with different levels of noise. 
The results are shown in \ref{fig:quad_con_con_ill_noise}. 
As can be seen from these results the conclusions from the deterministic experiments are still valid in the stochastic case.

\subsection{The Effect of the Graph's Connectivity} In our last set of experiments we show how the efficiency of the preconditioned dynamics in \eqref{eq: gauss-seidel zm} are essentially unaffected by the graph's connectivity. 
For this experiment we take an Erd{\H o}s-R{\`e}nyi graph with $N=100$ nodes and run EPISMD$-(\nabla^2\Phi,I)$ (i.e. no graph preconditioning) 
and EPISMD$-(\nabla^2\Phi,\nabla^2\Psi)$ (i.e. full preconditioning) on \eqref{eq:test problem} over the $d-$dimensional simplex and when the problem is well-conditioned.
We vary the probability $p$ of two nodes being connected (the parameter in Erd{\H o}s-R{\`e}nyi graphs) and run both algorithms for a fixed amount of CPU-time ($10$ seconds). 
For each probability $p$ we generated $20$ graphs and average the results.
The results are shown in \ref{fig:quad_con_erdos_ill_det}. 
As can be seen from the figure, the efficiency of EPISMD$-(\nabla^2\Phi,\nabla^2\Psi)$ is unaffected by the graph's connectivity; the same level of accuracy is reached independently of the probability $p$. 
On the other hand, the efficiency of EPISMD$-(\nabla^2\Phi,I)$ depends on the probability of two nodes being connected.
After a certain level $p\approx 0.5$ in this experiment we see that EPISMD$-(\nabla^2\Phi,\nabla^2\Psi)$ does not have a significant advantage over EPISMD$-(\nabla^2\Phi,I)$, and only performs marginally better, but still the extra cost of preconditioning seems justified even if the graph is well connected.

\begin{figure}
\captionsetup[subfigure]{justification=Centering}
\begin{subfigure}[t]{0.45\textwidth}
    \includegraphics[width=\textwidth]{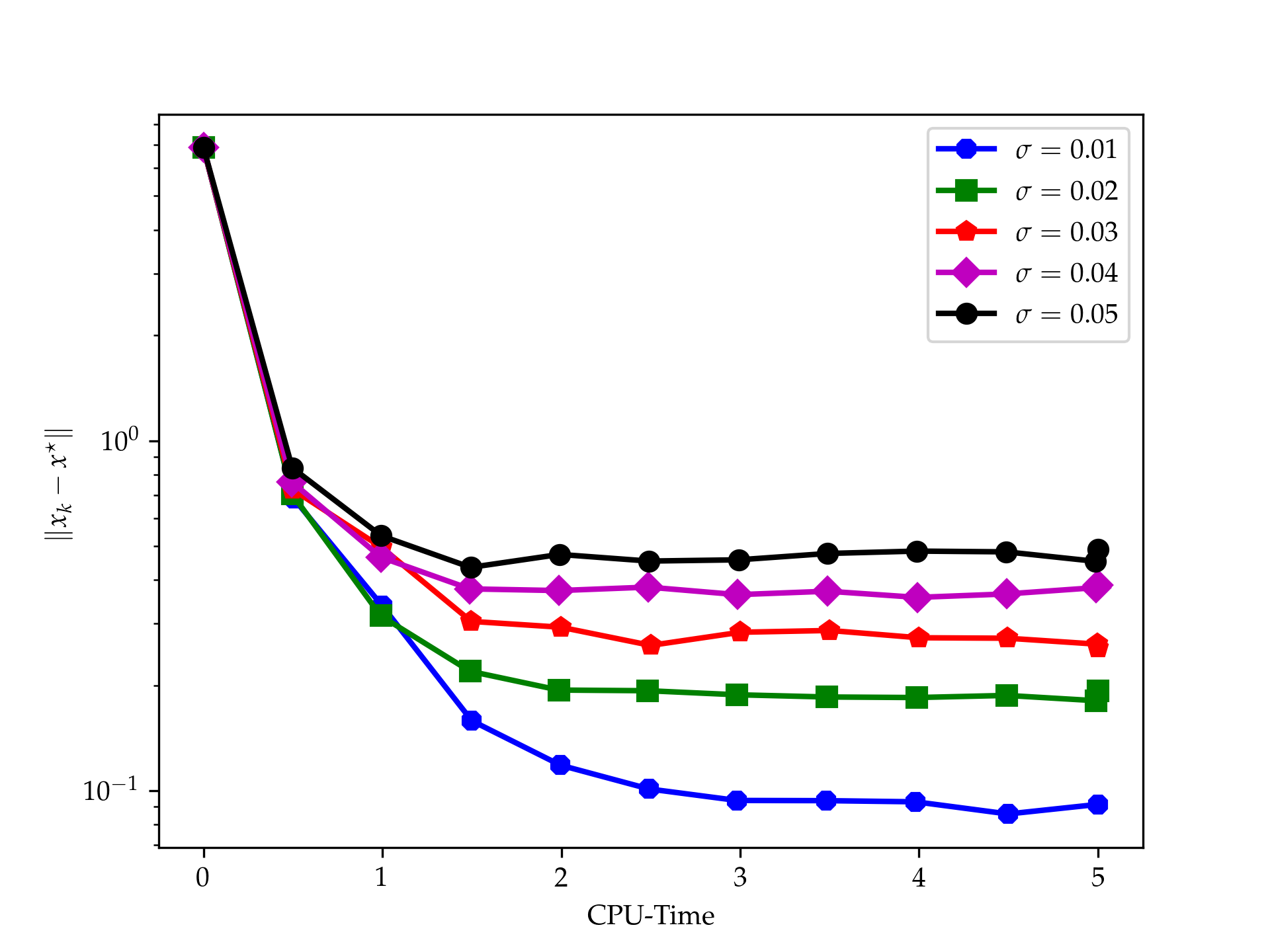}
    \caption{\label{fig:quad_con_con_ill_noise}}
\end{subfigure}\hspace{\fill} 
\begin{subfigure}[t]{0.45\textwidth}
    \includegraphics[width=\linewidth]{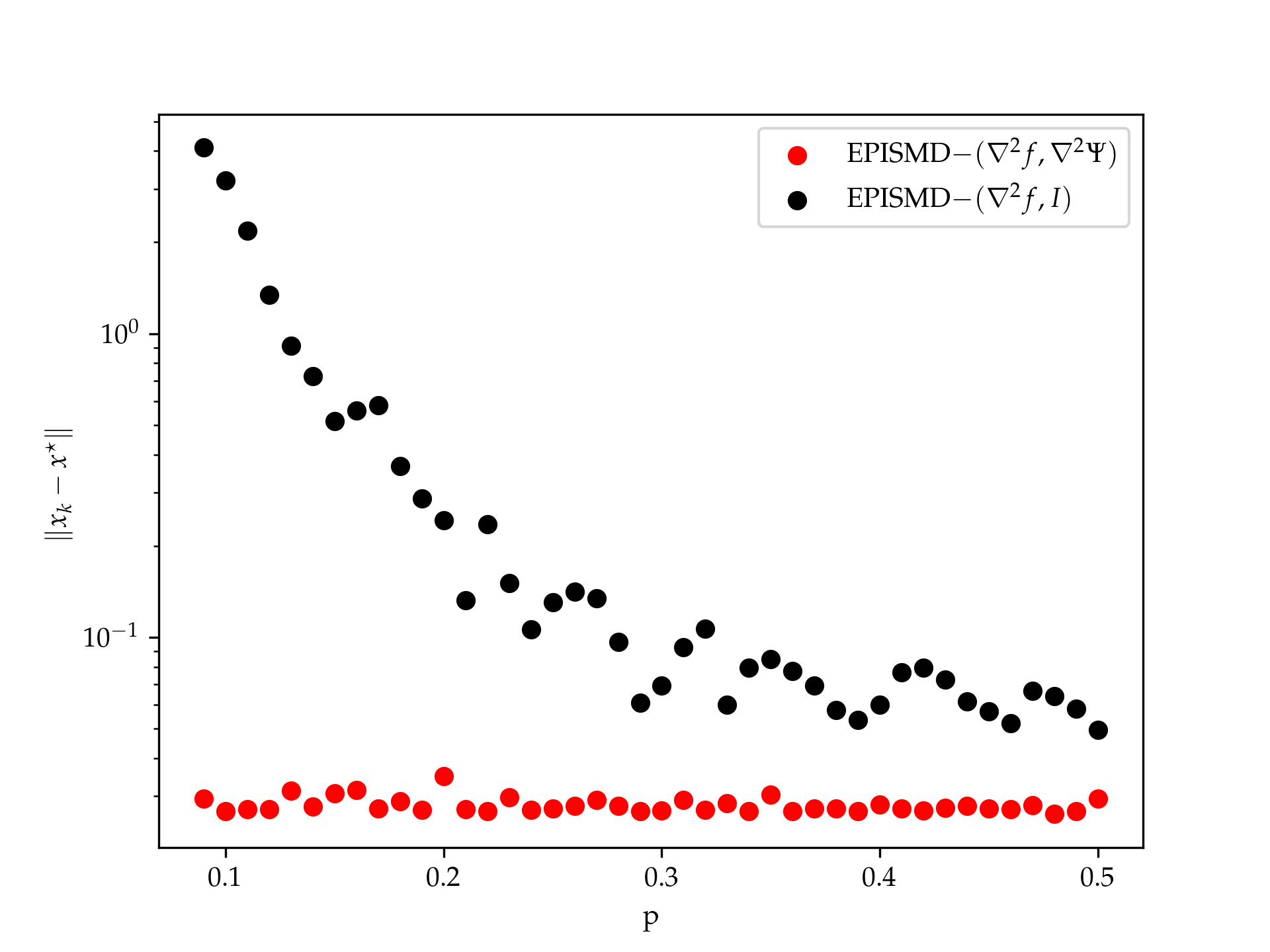}
    \caption{\label{fig:quad_con_erdos_ill_det}
    }
\end{subfigure}
\caption{
(a) The impact of different noise levels in a ring of cliques graph with a $d$-dimensional simplex constraint. (b) Experiment with different Erd{\H o}s-R{\`e}nyi graphs as we vary the probability $p$ that two nodes are connected (x-axis).
In this case we see that the convergence rate of the preconditioned algorithm EPISMD$-(\nabla^2f,\nabla^2\Psi)$ is independent of $p$. This is not the case if we only precondition the primal variables using EPISMD$-(\nabla^2 f,I)$.}
\end{figure} 

\section{Discussion}\label{sec:conclusions}
For future work several interesting extensions can be considered. 
One interesting direction is to study the proposed algorithm for sampling.
For example, the discretization we discussed in this paper will lead to a biased method and the amount of bias will depend on the discretization interval. 
In this sense exactness is lost, but could be recovered again using different methods from the sampling literature \cite{durmus2016sampling}. Addressing this in detail and comparing the differences between sampling and optimization in the context of distributed methods is an interesting direction for future work.

The analysis in this paper relies heavily on convexity. 
In preliminary numerical results not presented here EPISMD was used on a non-convex setting to train a small neural network with a negative entropy mirror map with promising performance. Future work can include a theoretical analysis of the non-convex case as well as exploring further the benefits of different mirror maps in the distributed setting. 
This is a challenging direction since the links with the dual are not exact. Nevertheless, the proposed preconditioner could be adapted to the non-convex setting using Levenberg–Marquardt regularization \cite{bazaraa2013nonlinear}. 
Finally, the proposed algorithm relies on second-order information. While this may seem prohibitive in practice, several recent works have shown that second order methods can be applied to problems where $d>10^6$ using sub-sampling and multilevel optimization techniques \cite{tsipinakis2021multilevel}. Jointly sampling second-order information from the Hessian and the graph is an interesting direction to consider for large scale models.

\paragraph{Acknowledgements}
This project was funded by JPMorgan Chase $\&$ Co under J.P. Morgan A.I. Research Awards in 2019 and 2021. 
G.A.P. was partially supported by the EPSRC through grant number EP/P031587/1.

\bibliographystyle{siam}
\bibliography{biblio}

\appendix
\section{Auxiliary results}\label{app:auxiliaryresults}

{
\paragraph{Proof of Lemma \ref{lemma: bound laplacian}}

Using the definition of the pseudo-inverse in \eqref{eq:pseudoinverse} and its relationship with the inverse of the regularized Laplacian in \eqref{eq:reglaplacianpseudo} we obtain,
\begin{equation}
\begin{split}
\inp{\bx}{\mathcal L\bx}&=\inp{\bx}{\mathcal L\mathcal L^+\mathcal L\bx}=\inp{\mathcal L\bx}{ (\cL^{-1}_\beta-\frac{1}{\beta N}\mathbf{1}_d\mathbf{1}_d^\top\otimes I_N) \mathcal L\bx }\\
&=\inp{\mathcal L\bx}{\cL^{-1}_\beta \mathcal L\bx},
\end{split}
\end{equation}
where in the last equality we used the fact that $(\mathbf{1}_d\mathbf{1}_d^\top\otimes I_N) \mathcal L=0$. Since $\cL_{\beta}\preceq \kappa_{\beta,N}I$ then 
$\cL^{-1}_{\beta}\succeq \kappa_{\beta,N}^{-1}I$ and the result follows.

\paragraph{Proof of Lemma \ref{lemma: postive k_g}}
We note that $A(\bx)$ can be obtained by removing the last $Nd$ columns and rows of the following matrix,
\begin{align}
B(\bx):=
\begin{bmatrix}
\nabla^2 f(\bx)+\mathcal L & \mathcal L \\
-\mathcal L & 0 
\end{bmatrix}.
\end{align}
Let $\bd=[\bd_x^\top, \bd_\lambda^\top]^\top$ and note that $\inp{\bd}{B(\bx)\bd}=\bd_x^\top\nabla^2 f(\bx)\bd_x$. It follows from the relative strong convexity assumption that $B(\bx)\succeq \mu_\Phi\nabla^2 \Phi(\bx)$ and therefore
$\|B(\bx)\|^2_{\nabla^2 \Phi(\bz)^{-1}}>0$.
Since $A(\bx)$ can be obtained by removing the last $Nd$ columns and rows of $B(\bx)$ the result follows from the interlacing theorem for singular values, see e.g. Theorem 3.1.3 in \cite{horn1994topics}.

\paragraph{Proof of Lemma \ref{lemma:bound on lagrangian and mirror map v1}}
Since $f$ is twice differentiable there exists an 
$\mathbf{y}$ on the line segment joining $\mathbf{x}$ and $\mathbf{x}^\star$ such that $\nabla f(\mathbf{x})-\nabla f(\mathbf{x}^\star) = \langle \nabla^2 f(\mathbf{y}), \mathbf{x}-\mathbf{x}^\star\rangle$. We then have,
\begin{align}
\|\nabla f(\bx)+\cL\blambda+\cL\bx\|^2_{\nabla^2\Phi^*(\bz)}
&= \|\nabla f(\bx)-\nabla f(\bx^\star)+\mathcal{L}(\blambda_t-\blambda^*)+\mathcal{L}(\mathbf{x}-\mathbf{x}^*)\|^2_{\nabla^2\Phi^*(\mathbf{z})}
\\
& = \|A(\mathbf{y})[\mathbf{x}_t-\mathbf{x}^*, 
\blambda-\blambda^*]^T\|^2_{\nabla^2\Phi^*(\mathbf{z})}\\
& \geq \kappa_g \left(
\|\mathbf{x}^*-\mathbf{x}_t\|^2+\|\blambda^*-\blambda\|^2
\right)\\
&\geq \frac{2\kappa_g}{\hat\mu}
\left(\sum_{i=1}^N D_\Phi(x^\star,x^i)+D_\Psi(\lambda^\star,\lambda^i)
\right).
\end{align}
We use \eqref{eq:breg7} to obtain the bound in terms of the dual variables.

}

\color{black}
\subsection{When first-order optimization fails}\label{sec:first order fails}
Our first result shows that if there exists an $x^\star$ such that $\nabla f_i(x^\star)=0$ for all  $i=1,...,N$, exact consensus can be obtained for ISMD.

\begin{lemma}\label{lem:exact_convergence}
Let Assumptions \ref{ass:f}-\ref{ass:mirror1} hold. 
Consider the dynamics in \eqref{eq:cmdpart_vec} with $\sigma=0$. 
If
\begin{align}\label{eq:cond conv}
    \bigcap_{i=1}^N\{\nabla f_i(x)=0\}\neq \emptyset,
\end{align}
then $\lim\limits_{t\rightarrow\infty} x_t^i = x^\star,$ where $x^\star$ is an optimal point for \eqref{eq:distoptobj0} such that $x^\star=\min\limits_{x\in\mathcal X^\star} D_{\Phi}(x,x_0)$.

\end{lemma}
\begin{proof}
Let $x_0$ be the initial point of the algorithm, and let $z^\star$ be an optimal (dual) point closest to $z_0=\nabla \Phi(x_0)$ with respect to the divergence generated by $\Phi^*$, 
\[
z^\star=\argmin_{z\in Z^\star} D_{\Phi^*}(z,z_0)
\] 
where $Z^\star=\{z~|~ z=\nabla \Phi(x),\exists x \in \mathcal X: \nabla f_i(x)=0, i=1,\ldots,N\}$. 
By assumption \eqref{eq:cond conv}, $Z^\star$ is not empty. With a slight abuse of notation we let $x^\star=\nabla \Phi^*(z^\star)$ and note that $(x^\star,z^\star)$ is an equilibrium
point for \eqref{eq:cmdpart_vec} (for a strongly convex function it is also the unique equilibrium point, but here we only assume convexity of $f$). 

Define the Lyapunov candidate function $V_t = \sum_{i=1}^N D_{\Phi^\star}(z_t^i,z^\star)$, and note that given our assumptions on $f$, and $\Phi$ it follows that $V$ is a proper function (has compact sub-level sets). Then we obtain,
\begin{align}
    dV_t &= \sum_{i=1}^N(x^\star-x_t^i)^T\nabla f_i(x_t^i)dt +\sum_{i=1}^N (x_t^i-x^\star)^T\sum_{j=1}^N A_{ij}(z_t^j-z_t^i)dt.
\end{align}
Under 
convexity of $f$ and optimality at $x^\star$ we have
\begin{align}\label{eq:fpropconv}
    \sum_{i=1}^N(x^\star-x_t^i)^T\nabla f_i(x_t^i) \leq \sum_{i=1}^N (f_i(x^\star)-f_i(x_t^i))\leq 0.
\end{align}
By the triangle equality of the Bregman divergence in \eqref{eq:triangle_bregman},
\begin{align}
    (x_t^i-x^\star)^T(z_t^j-z_t^i)&=-(x^\star-x_t^i)^T(z_t^j-z_t^i)\\
    &=-(\nabla\Phi^*(z^\star)-\nabla\Phi^*(z_t^i))^T(z_t^j-z_t^i)\\
    &= -D_{\Phi^*}(z_t^j,z_t^i)-D_{\Phi^*}(z_t^i,z^\star)+D_{\Phi^*}(z_t^j,z^\star).
\end{align}
Then,
\begin{align}\label{eq:negativitysum}
    \sum_{i=1}^N\sum_{j=1}^N&A_{ij}(x_t^i-x^\star)^T(z_t^j-z_t^i)\\
    &=\sum_{i=1}^N\sum_{j=1}^NA_{ij}\left(-D_{\Phi^*}(z_t^j,z_t^i)-D_{\Phi^*}(z_t^i,z^\star)+D_{\Phi^*}(z_t^j,z^\star)\right)\leq 0,
\end{align}
where we used $A_{ij}\geq 0$, $\sum_{i=1}^N\sum_{j=1}^N D_{\Phi^*}(z_t^i,z^\star) = \sum_{i=1}^N\sum_{j=1}^ND_{\Phi^*}(z_t^j,z^\star)$, and $D_{\Phi^*}(z_t^j,z_t^i)\geq 0$.
Since $V_t>0$ for $\mathbf{z}\neq \mathbf 1_N \otimes z^\star$, $V_t=0$ when $\mathbf{z}=\mathbf 1_N \otimes z^\star$ and $d V_t \leq 0$ with equality only at $\mathbf{z}=\mathbf 1_N \otimes z^\star$ we conclude that $V_t$ is a Lyapunov function for $\mathbf{z}_t$. Also note that since $D_{\Phi^*}(z_t^i,z^\star)=D_{\Phi}(x^\star,x_t^i)$ $V_t$ is also Lyapunov function for $\mathbf{x}_t$. Since all the Lyapunov stability criteria are satisfied (see e.g. c\cite[Theorem 15.4]{bullo2019lectures}) it follows that $x^\star$ is globally asymptotically stable for the dynamics in \eqref{eq:cmdpart_vec}.
\end{proof}

The Lemma above can be extended to an if and only if statement based on the arguments of \cite[Theorem 1]{shi2015network}, but precise details lie beyond the scope of this paper.  If  \eqref{eq:cond conv} is violated, even with the right choice of mirror map, achieving exact consensus is not possible. In general imposing $x^\star$ to satisfy \eqref{eq:cond conv} is quite restrictive as $\nabla f(\bx)=\sum_{i=1}^N\nabla f_i(x^\star)=0$ does not necessarily imply $\nabla f_i(x^\star)=0$ for all $i=1,...,N$. The crucial point to realize here is that if and only if \eqref{eq:cond conv} holds then $(\bx^\star,\bz^\star$) will also be the minimizer of $f(\bx)+\frac{1}{2}\bz^T\mathcal L \bz$; see \cite[Lemma 7]{shi2015network} for details. As a result one can establish consensus at equilibrium and $V_t$ will approach zero at large $t$. 

\color{black}



If $\bx^\star$ does not satisfy \eqref{eq:cond conv} and one has just $\nabla f(\bx^\star)=0$, even in the deterministic case exact consensus and optimality at convergence can no longer be achieved. But the arguments above can be used to establish exponential convergence of \eqref{eq:cmdpart_vec} to a neighborhood of $(\bx^\dagger,\bz^\dagger)$ both minimizing $f(\bx)+\frac{1}{2}\bz^T\mathcal L \bz$. The size of this neighborhood naturally depends on the noise and $f$.

\begin{proposition}[Approximate convergence of \eqref{eq:cmdpart_vec}]\label{prop:approx_conv_app}
Let Assumptions \ref{ass:f}-\ref{ass:mirror1} hold and assume that $f$ is $\mu_f$-strongly convex w.r.t. $\Phi$. Let $\bx^\dagger=\arg\min \{f(\bx)+\frac{1}{2}\nabla\Phi(\bx)^T\mathcal L \nabla\Phi(\bx)\}$ with $\bx^\dagger = 1_N\otimes x^\dagger$ and $V_t = \frac{1}{N}\sum_{i=1}^N D_{\Phi^*}(z_t^i,z^\dagger)$, where $z_t^i$ obeys the dynamics of \eqref{eq:cmdpart} (or \eqref{eq:cmdpart_vec}). Then we have:
\begin{align}
    \mathbb{E}\left[V_t\right]\leq e^{-\eta\mu_f t} \frac{1}{N}\sum_{i=1}^ND_{\Phi^*}(z_0^i,z^\dagger)+\frac{\sigma^2}{2\eta\mu_f}||\Delta\Phi^*||_\infty+\frac{1}{ \mu_f}\left(f(\bx^\dagger)-f(\bx^\circ)\right),
\end{align}
where $x^\circ=\arg\min_{\bx\in\mathcal X} f(\bx)$.
\end{proposition}

\begin{proof}
Denote $\bx^\dagger=\arg\inf\{f(\bx)+\frac{1}{2}\nabla\Phi(\bx)^T\mathcal L \nabla \Phi(\bx)\}$, $\bz^\dagger=\nabla\Phi(\bx^\dagger)$. Then
\begin{align}
    dV_t =& -\frac{1}{N}\sum_{i=1}^N(x_t^i-x^\dagger)^T\eta \nabla f_i(x_t^i)dt +\epsilon\frac{1}{N}\sum_{i=1}^N (x_t^i-x^\dagger)^T\sum_{j=1}^NA_{ij}(z_t^j-z_t^i)dt \\\label{eq:v_t_ismd}
    &+ \frac{1}{2}\sigma^2\frac{1}{N}\sum_{i=1}^N \text{tr}(\Delta\Phi^*(z_t^i)) dt+ \frac{1}{N}\sum_{i=1}^N (x_t^i-x^\dagger)^TdB_t^i.
\end{align}
Using the $\mu_f$-strong convexity of $f$ we obtain,
\begin{align}
\nabla f(\bx_t)^T(\bx^\dagger-\bx_t)\leq & f(\bx^\dagger) -f(\bx_t) -\mu_f D_{\phi}(\bx^\dagger,\bx_t) \\
\leq & f(\bx^\dagger)-f(\bx^\circ)+f(\bx^\circ)-f(\bx_t) -\mu_f D_{\phi}(\bx^\dagger,\bx_t) 
\\
\leq & f(\bx^\dagger)-f(\bx^\circ) -\mu_f D_{\phi}(\bx^\dagger,\bx_t)  
\end{align}
Substituting in \eqref{eq:v_t_ismd} and using \eqref{eq:breg7}, \eqref{eq:negativitysum} and taking expectations gives
\begin{align} \label{eq:bound-V-1}
    \frac{d\mathbb{E}[V_t]}{dt} \leq - \eta\mu_f \mathbb{E}[V_t] + \eta\left(f(\bx^\dagger)-f(\bx^\circ)\right)+ \frac{1}{2}\sigma^2 ||\Delta\Phi^*||_\infty.
\end{align}
Standard Gr\"onwall arguments gives the result.
\end{proof}

While the relative strong convexity of the objective function can speed up convergence, even in the no noise setting the above bound does not guarantee that the algorithm is able to converge exactly to the optimum. In this setup the additional preconditioning via the mirror map does not facilitate exact convergence nor consensus. When \eqref{eq:cond conv} holds, then $\bx^\circ$, $\bx^\star$ and $\bx^\star$ coincide \cite[Lemma 7]{shi2015network}, so $f(\bx^\dagger)-f(\bx^\circ)=0$.

\end{document}